\documentclass[11pt]{article}
\usepackage{amsmath,amsthm,amsfonts,amssymb,bm,wasysym}
\usepackage{epsfig}
\usepackage[usenames]{color}
\usepackage{verbatim}
\usepackage{hyperref}
\usepackage{multicol}
\usepackage{comment}
\usepackage{float}
\usepackage{graphicx}
\usepackage{centernot}

\usepackage[colorinlistoftodos]{todonotes}
\usepackage[normalem]{ulem}

\parskip \medskipamount
\newtheorem{theorem}{Theorem}[section]
\newtheorem{definition}[theorem]{Definition}

\numberwithin{equation}{section}
\newtheorem{lemma}[theorem]{Lemma}
\newtheorem{proposition}[theorem]{Proposition}
\newtheorem{corollary}[theorem]{Corollary}

\newtheorem{remark}[theorem]{Remark}

\numberwithin{equation}{section}

\def\N{\mathbb{N}}
\def\Z{\mathbb{Z}}

\def\C{\mathbb{C}}

\def\S{\mathcal{S}}

\def\F{\mathcal{F}}

\def\B{\mathcal{B}}
\def\LL{\mathcal{L}}

\renewcommand{\phi}{\varphi}
\renewcommand{\epsilon}{\varepsilon}

\def\GG{\mathcal{G}}
\def\tt{t(n)}

\allowdisplaybreaks

\newcommand{\1}{{\text{\Large $\mathfrak 1$}}}

\renewcommand{\emptyset}{\varnothing}

\def\C{{\mathcal C}}

\newcommand{\til}{\widetilde}

\newcommand{\tstop}{t_{\mathrm{stop}}}

\newcommand{\pr}[1]{\mathbb{P}\!\left(#1\right)}
\newcommand{\E}[1]{\mathbb{E}\!\left[#1\right]}
\newcommand{\estart}[2]{\mathbb{E}_{#2}\!\left[#1\right]}
\newcommand{\prstart}[2]{\mathbb{P}_{#2}\!\left(#1\right)}
\newcommand{\prcond}[3]{\mathbb{P}_{#3}\!\left(#1\;\middle\vert\;#2\right)}

\newcommand{\escond}[3]{\mathbb{E}_{#3}\!\left[#1\;\middle\vert\;#2\right]}

\newcommand{\estarth}[2]{\mathbb{\widehat{E}}_{#2}\!\left[#1\right]}

\newcommand{\escondh}[3]{\mathbb{\widehat{E}}_{#3}\!\left[#1\;\middle\vert\;#2\right]}

\newcommand{\norm}[1]{\left\| #1 \right\|}

\newcommand{\tn}{|\kern-.1em|\kern-0.1em|}

\newcommand\be{\begin{equation}}
\newcommand\ee{\end{equation}}

\def\eps{\varepsilon}

\newcommand{\tv}[1]{\left\|#1\right\|_{\rm{TV}}}

\newcommand{\tmix}[2]{t_{\mathrm{mix}}(\epsilon, #1, #2)}
\newcommand{\tx}[1]{t_{\mathrm{mix}}(\epsilon, #1)}

\newcommand{\cps}[1]{\mathcal{P}_{\eta_0}\!\left(#1\right)}

\newcommand{\tmixx}[2]{t_{\mathrm{mix}}(#1, #2)}

\newcommand{\diam}[1]{{\rm{diam}}(#1)}

\begin{document}

\title{Mixing time for random walk on supercritical dynamical percolation}

\author{Yuval Peres\thanks{Microsoft Research, Redmond WA, U.S.A.\ \ Email:
        \hbox{peres@microsoft.com}} \and Perla Sousi\thanks{University of Cambridge, Cambridge, UK.\ \ Email: \hbox{p.sousi@statslab.cam.ac.uk}} \and
        Jeffrey E. Steif\thanks{Chalmers University of Technology
and Gothenburg University, Gothenburg, Sweden.\ \ Email:
        \hbox{steif@chalmers.se}}
}

\maketitle
\thispagestyle{empty}

\begin{abstract}
We consider dynamical percolation on the $d$-dimensional discrete torus of side length $n$, $\Z_n^d$, where each edge refreshes its status at rate $\mu=\mu_n\le 1/2$ to be open with probability $p$. We study random walk on the torus, where the walker moves at rate $1/(2d)$ along each open edge. In earlier work of two of the authors with A. Stauffer, it was shown that in the subcritical case $p<p_c(\Z^d)$, the (annealed) mixing time of the walk is $\Theta(n^2/\mu)$, and it was conjectured that in the supercritical case $p>p_c(\Z^d)$, the mixing time is $\Theta(n^2+1/\mu)$; here the implied constants depend only on~$d$ and $p$.  We prove a quenched (and hence annealed) version of this conjecture up to a poly-logarithmic factor under the assumption $\theta(p)>1/2$.  Our proof is based on percolation  results (e.g., the Grimmett-Marstrand Theorem) and an analysis of the volume-biased evolving set process; the key point is that typically, the evolving set has a substantial intersection with the giant percolation cluster at many times. This allows us to use precise isoperimetric properties of the cluster (due to G. Pete) to infer rapid growth of the evolving set, which in turn yields the upper bound on the mixing time.

 \newline
 MSC 2010 \emph{subject classifications.}
 Primary 60K35, 60K37
  \medskip\noindent
\end{abstract}

\section{Introduction}

This paper studies random walk on dynamical percolation on the torus $\Z^d_n$. The edges refresh at rate $\mu\leq 1/2$ and switch to open with probability $p$ and closed with probability $1-p$ where $p>p_c(\Z^d)$ with~$p_c(\Z^d)$ being the critical probability for bond percolation on $\Z^d$. The random walk moves at rate $1$. When its exponential clock rings, the walk chooses one of the $2d$ adjacent edges with equal probability. If the bond is open, then it makes the jump, otherwise it stays in place. 

We represent the state of the system by $(X_t,\eta_t)$, where $X_t\in \Z_n^d$ is the location of the walk at time~$t$ and $\eta_t\in \{0,1\}^{E(\Z_n^d)}$ is the configuration of edges at time $t$, where $E(\Z_n^d)$ stands for the edges of the torus. We emphasise at this point that $(X_t,\eta_t)$ is Markovian, while the location of the walker~$(X_t)$ is not.

One easily checks that $\pi\times \pi_p$ is the unique stationary distribution
and that the process is reversible; here $\pi$ is uniform distribution and
$\pi_p$ is product measure with density $p$ on the edges. Moreover, if the environment
$\{\eta_t\}$ is fixed, then $\pi$ is a stationary distribution for the resulting time
inhomogeneous Markov process.

This model was introduced by Peres, Stauffer and Steif in~\cite{PerStaufSteif}. We emphasise that $d$ and $p$ are considered fixed, while $n$ and $\mu=\mu_n$ are the two parameters which are varying. The focus of~\cite{PerStaufSteif} was to study the total variation mixing time of $(X,\eta)$, i.e.
\[
t_{\rm{mix}}(\epsilon) = \min\left\{t\geq 0: \max_{x,\eta_0}\norm{\prstart{(X_t,\eta_t) = (\cdot, \cdot)}{x,\eta_0}-\pi\times \pi_p}_{\rm{TV}}\leq \epsilon\right\}.
\]
They focused on the subcritical regime, i.e.\ when $p<p_c$ and they proved the following: 
\begin{theorem}[\cite{PerStaufSteif}]
	For all $d\geq 1$ and $p<p_c$ the mixing time of $(X,\eta)$ satisfies
	\[
	t_{\rm{mix}}(1/4) \asymp \frac{n^2}{\mu}.
	\]
\end{theorem}

They also established the same mixing time when one looks at the walk and averages over the environment.

In the present paper we focus on the supercritical regime. We study both the full system and the quenched mixing times. We start by defining the different notions of mixing that we will be using. First of all we write $\prstart{\cdot}{x,\eta}$ for the probability measure of the walk, when the environment process is conditioned to be $\eta = (\eta_t)_{t\geq 0}$ and the walk starts from $x$. We write $\mathcal{P}$ for the distribution of the environment which is dynamical percolation on the torus, a measure on c\`adl\`ag paths $[0,\infty) \mapsto\{0,1\}^{E(\Z^d_{n})}$. We write $\mathcal{P}_{\eta_0}$ to denote the measure $\mathcal{P}$ when the starting environment is $\eta_0$. Abusing notation we write $\prstart{\cdot}{x,\eta_0}$ to mean the law of the full system when the walk starts from $x$ and the initial configuration of the environment is $\eta_0$. To distinguish it from the quenched law, we always write $\eta_0$ in the subscript as opposed to $\eta$.

 For $\epsilon \in (0,1)$, $x\in\Z_n^d$ and a fixed environment $\eta=(\eta_t)_{t\geq 0}$ we write
\[
\tmix{x}{\eta} = \min\left\{ t\geq 0:\tv{\prstart{X_t=\cdot}{x,\eta} - \pi} \leq \epsilon \right\}.
\]
We also write 
\[
\tx{\eta} = \max_x \,\tmix{x}{\eta}
\]
for the quenched $\epsilon$-mixing time.
We remark that $\tx{\eta}$ could be infinite.  Using the obvious definitions,
the standard inequality $t_{\rm{mix}}(\epsilon)\le \log_2(1/\epsilon) t_{\rm{mix}}(1/4)$ does not
hold for time-inhomogeneous Markov chains and therefore also not for quenched mixing times.
Therefore, in such situations, to describe the rate of convergence to stationarity, it
is more natural to give bounds on~$t_{\rm{mix}}(\epsilon, \eta)$ for all $\epsilon$ rather than just considering $\epsilon=1/4$ as is usually done.

We first recall the result from~\cite{quenched-sub} which is an upper bound on the quenched mixing time and the hitting time of large sets for all values of $p$. We write $\tau_A$ for the first time $X$ hits the set $A$.

\begin{theorem}[\cite{quenched-sub}]\label{thm:recall}
For all $d\ge 1$ and $\delta>0$, there exists 
$C=C(d,\delta)<\infty$ so that for all $p\in [\delta,1-\delta]$, for all $\mu \le 1/2$, for all $n$ and for all $\eps$, random walk in dynamical percolation 
on $\Z_n^d$ with parameters~$p$ and $\mu$ satisfies for all $x$
\begin{equation}\label{eq:cor:supMeanSquaredDisplacement}
\max_{\eta_0}\cps{\eta = (\eta_t)_{t\geq 0}: \,\tmix{x}{\eta}\ge \frac{Cn^2\log(1/\epsilon)}{\mu^4}}\leq
\epsilon.
\end{equation}
Moreover, for all $A\subseteq \Z_n^d$ with $|A|\geq n^d/2$ we have
\begin{align}\label{eq:hitbigset}
\nonumber\max_{\eta_0}\,&\cps{\eta=(\eta_t)_{t\geq 0}: \, \max_x\estart{\tau_A}{x,\eta} \leq \frac{Cn^2\log(n/\epsilon)}{\mu}} \leq \epsilon \text{ and}\\
&\max_{x,\eta_0} \estart{\tau_A}{x,\eta_0}\leq \frac{Cn^2}{\mu}.
\end{align}
\end{theorem}

Our first result concerns the quenched mixing time in the case when $\theta(p)>1/2$. 

\begin{theorem}\label{thm:qmixlarge}
	Let $p\in (p_c(\Z^d), 1)$ with $\theta(p)>1/2$ and $\mu\leq 1/2$. Then there exists $a>0$ (depending only on $d$ and $p$) so that for all $n$ sufficiently large we have
	\[
	\sup_{\eta_0}\cps{\eta= (\eta_t)_{t\geq 0}: \, t_{\rm{mix}}\left(n^{-3d}, \eta\right) \geq (\log n)^a \left( n^2 + \frac{1}{\mu}\right) } \leq \frac{1}{n^d}.
	\]
\end{theorem}

\begin{remark}\label{rem:13}
\rm{
We note that when $1/\mu<(\log n)^b$ for some $b>0$, then Theorem~\ref{thm:qmixlarge} follows from Theorem~\ref{thm:recall}. (Take $\epsilon=n^{-3d}$ in~\eqref{eq:cor:supMeanSquaredDisplacement} and do a union bound over $x$.) So we are going to prove Theorem~\ref{thm:qmixlarge} in the regime when $1/\mu> (\log n)^{d+2}$.}	
\end{remark}

Our second result concerns the mixing time at a stopping time in the quenched regime for all values of $p>p_c(\Z^d)$. We first give the definition of this notion of mixing time that we are using.


\begin{definition}\rm{
For $\epsilon\in (0,1)$ and a fixed environment $\eta=(\eta_t)_{t\geq 0}$ we define
\[
\tstop(\epsilon,\eta) = \max_{x} \inf\left\{ \estart{T}{x,\eta}: \, T \text{ randomised stopping time s.t.} \ \norm{\prstart{X_T=\cdot}{x,\eta} - \pi}_{\rm{TV}}  \leq \epsilon   \right\}.
\]
}	
\end{definition}

\begin{theorem}\label{thm:qmixcesaro}
Let $p\in (p_c(\Z^d),1)$, $\epsilon>0$ and $\mu\leq 1/2$. Then there exists $a>0$ (depending only on $d$ and $p$) so that for all $n$ sufficiently large we have 
\[
\inf_{\eta_0} \mathcal{P}_{\eta_0}\left( \eta=(\eta_t)_{t\geq 0}: \, \tstop(\epsilon,\eta) \leq (\log n)^a \left(n^2+\frac{1}{\mu} \right)   \right) = 1-o(1).
\]
\end{theorem}

%
%
%

Finally we give a consequence for random walk on dynamical percolation on all of~$\Z^d$. This is defined analogously to the process on the torus $\Z_n^d$ above, where the edges refresh at rate $\mu$.

\begin{corollary}\label{cor:hittingtimes}
Let $p\in (p_c(\Z^d),1)$ and $\mu\leq 1/2$. Let $X$ be the random walk on dynamical percolation on $\Z^d$ and  for $r>0$ let
\[
\tau_r = \inf\{t\geq 0: \, \norm{X_t}\geq r\}.
\]	
Then there exists $a>0$ (depending only on $d$ and $p$) so that for all $r$ 
\[
\sup_{\eta_0}\estart{\tau_r}{0,\eta_0} \leq \left(r^2+\frac{1}{\mu} \right)(\log r)^a.
\] 
\end{corollary}

\emph{Notation} For positive functions $f,g$ we write $f\sim g$ if $f(n)/g(n)\to 1$ as $n\to\infty$. We also write $f(n) \lesssim g(n)$ if there exists a constant $c \in (0,\infty)$ such that $f(n) \leq c g(n)$ for all $n$,
and $f(n) \gtrsim g(n)$ if $g(n) \lesssim f(n)$.  Finally, we use the notation $f(n) \asymp g(n)$ if both $f(n) \lesssim g(n)$ and $f(n) \gtrsim g(n)$. 

\textbf{Related work}
Various references to random walk on dynamical percolation has been provided in~\cite{PerStaufSteif}. In a very recent paper, Andres, Chiarini, Deuschel and Slowik~\cite{SebAn} have obtained a quenched invariance principle for random walks with time-dependent ergodic degenerate weights. Their framework though does not cover dynamical percolation, since the conductances are assumed to take strictly positive values.

\subsection{Overview of the proof}

In this subsection we explain the high level idea of the proof and also give the structure of the paper. First we note that when we fix the environment to be~$\eta$, we obtain a time inhomogeneous Markov chain. To study its mixing time, we use the theory of evolving sets developed by Morris and Peres adapted to the inhomogeneous setting, which was done in~\cite{quenched-sub}. We recall this in Section~\ref{sec:evolve}. 
 In particular we state a theorem by Diaconis and Fill that gives a coupling of the chain with the Doob transform of the evolving set. (Diaconis and Fill proved it in the time homogeneous setting, but the adaptation to the time inhomogeneous setting is straightforward.) The importance of the coupling is that conditional on the Doob transform of the evolving set up to time $t$, the random walk at time $t$ is uniform on the Doob transform at the same time. This property of the coupling is going to be crucial for us in the proofs of Theorems~\ref{thm:qmixlarge} and~\ref{thm:qmixcesaro}. 

The size of the Doob transform of the evolving set in the inhomogeneous setting is again a submartingale, as in the homogeneous case. The crucial quantity we want to control is the amount by which its size increases. This increase will be large only at \emph{good times}, i.e.\ when the intersection of the Doob transform of the evolving set with the giant cluster is a substantial proportion of the evolving set. Hence we want to ensure that there are enough \emph{good times}.  We would like to emphasise that in this case we are using the random walk to infer properties of the evolving set. More specifically, in Section~\ref{sec:hitgiant} we give an upper bound on the time it takes the random walk to hit the giant component. Using this and the coupling of the walk with the evolving set, in Section~\ref{sec:goodtimes} we establish that there are enough good times.
We then employ a result of G\'abor Pete which states that the isoperimetric profile of a set in a supercritical percolation cluster coincides with its lattice profile. We apply this result to the sequence of good times, and hence obtain a good drift for the size of the evolving set at these times.

We conclude Section~\ref{sec:goodtimes} by giving an upper bound that there exists a stopping time bounded by  with high probability so that at this time the Doob transform of the evolving set has size at least $(1-\delta) (\theta(p)-\delta) n^d$. In the case when $\theta(p)>1/2$ we can take $\delta>0$ sufficiently small so that~$(1-\delta) (\theta(p)-\delta)>1/2$. Using the uniformity of the walk on the Doob transform of the evolving set again, we deduce that at this stopping time the walk is close to the uniform distribution in total variation with high probability. This yields Theorem~\ref{thm:qmixlarge}.

To finish the proof of Theorem~\ref{thm:qmixcesaro} the idea is to repeat the above procedure to obtain $k$ sets whose union covers at least $1-\delta$ of the whole space. Then we define $\tau$ by choosing one of these times uniformly at random. At time $\tau$ the random walk will be uniform on a set with measure at least~$1-\delta$, and hence this means that the total variation from the uniform distribution at this time is going to be small. Since this time is with high probability smaller than $k$ times the mixing time, this finishes the proof.

 \section{Evolving sets for inhomogeneous Markov chains}\label{sec:evolve}

 In this section we give the definition of the evolving set process for a discrete time inhomogeneous Markov chain.
 
Given a general transition matrix $p(\cdot, \cdot)$ with state space $\Omega$ and stationary distribution $\pi$ we let for~$A,B\subseteq \Omega$
\begin{equation}\label{eq:DefnQ}
Q_{p}(A,B):=\sum_{x\in A,y\in B} \pi(x)p(x,y).
\end{equation}
When $B=\{y\}$ we simply write $Q_p(A,y)$ instead of $Q_p(A,\{y\})$.

We first recall the definition of evolving sets in the context of a finite state discrete time Markov chain with state space $\Omega$, transition matrix $p(x,y)$ and stationary 
distribution $\pi$.  The evolving-set process $\{S_n\}_{n\ge 0}$ is a Markov chain
on subsets of $\Omega$ whose
transitions are described as follows. Let $U$ be
a uniform random variable on $[0,1]$. If $S\subseteq \Omega$ is the present state, we let the next 
state $\til{S}$ be defined by
$$
\til{S}:=\left\{y\in \Omega: \frac{Q_p(S,y)}{\pi(y)}\ge U\right\}.
$$
We remark that $Q_p(S,y)/\pi(y)$ is the probability that the reversed chain starting at $y$ is in $S$ after one step.
Note that $\Omega$ and $\emptyset$ are absorbing states and it is immediate to check that
\begin{equation}\label{eq:ES1d}
\prcond{y\in S_{k+1}}{S_k}{}=\frac{Q_p(S_k,y)}{\pi(y)}.
\end{equation}
Moreover, one can describe the evolving set process as that process on subsets which 
satisfies the ``one-dimensional marginal'' condition~\eqref{eq:ES1d} and where these 
different events, as we vary $y$, are maximally coupled.

For a transition matrix $p$ with stationary distribution $\pi$ we define for $S$ with $\pi(S)> 0$
\begin{align*}
\phi_{p}(S):=\frac{Q_{p}(S,S^c)}{\pi(S)} \quad \text{ and } \quad \psi_{p}(S):=1-\E{\sqrt{\frac{\pi(\til{S})}{\pi(S)}}},
\end{align*}
where $\til{S}$ is the first step of the evolving set process started from $S$ when the transition probability for the Markov chain is $p$ and as always the stationary distribution is $\pi$.

For $r\in [\min_x\pi(x), 1/2]$ we define $\psi_p(r) :=\inf\{\psi_p(S): \pi(S)\leq r\}$ and $\psi_p(r)=\psi_p(1/2)$ for $r\geq 1/2$. We define $\phi_p(r)$ analogously. We now recall a lemma from Morris and Peres~\cite{MorrisPeres} that will be useful later.

\begin{lemma}[{\cite[Lemma~10]{MorrisPeres}}]\label{lem:phipsi}
	Let $0<\gamma\leq 1/2$ and let $p$ be a transition matrix on the finite state space $\Omega$ with $p(x,x)\geq \gamma$ for all $x$. Let $\pi$ be a stationary distribution. Then for all sets $S\subseteq \Omega$ with $\pi(S)>0$ we have
	$$
1-\psi_{p}(S)\le 1-\frac{\gamma^2}{2(1-\gamma)^2}\cdot (\varphi_{p}(S))^2.
$$
\end{lemma}

We next define completely analogously to the time homogeneous case the evolving set process in the context of a time inhomogeneous Markov chain with a stationary distribution~$\pi$. Consider a time inhomogeneous Markov chain with state space
$\S$ whose transition matrix for moving from time $k$ to time $k+1$ 
is given by $p_{k+1}(x,y)$ where we assume that the 
probability measure $\pi$ is a stationary distribution for each $p_k$. 
In this case, we say that~$\pi$ is a stationary distribution for the inhomogeneous Markov chain.
Let~$Q_k= Q_{p_k}$ be as defined 
in~\eqref{eq:DefnQ}. We then obtain a time 
inhomogeneous Markov chain $S_0,S_1,\ldots$ on subsets of $\S$ generated by 
$$
{S}_{k+1}:=\left\{y\in \S: \frac{Q_{k+1}(S_k,y)}{\pi(y)}\ge U\right\}
$$
where $U$ is as before a uniform random variable on $[0,1]$. We call this the evolving set process
with respect to $p_1,p_2,\ldots$ and stationary distribution $\pi$. 

We now define the Doob transform of the evolving set process associated to a time inhomogeneous Markov chain. If $K_p$ is the transition probability for the evolving set process when the transition matrix for the Markov chain is $p$, then we define the Doob transform with respect to being absorbed at $\Omega$ via
\[
\widehat{K}_p(S,S') = \frac{\pi(S')}{\pi(S)} K_p(S,S').
\] 
The following coupling of the time inhomogeneous Markov chain with the Doob transform of the evolving set will be crucial in the rest of the paper. The proof is identical to the proof of the homogeneous setting by Diaconis and Fill~\cite{DiaconisFill}. For the proof see for instance~\cite[Theorem~17.23]{LevPerWil}.

\begin{theorem}\label{thm:coupling}
	Let $X$ be a time inhomogeneous Markov chain. Then there exists a Markovian coupling of $X$ and the Doob transform $(S_t)$ of the associated evolving sets so that for all starting points $x$ and all times~$t$ we have $X_0=x$, $S_0=\{x\}$ and for all $w$
	\[
	\prcond{X_t=w}{S_0,\ldots, S_t}{x} = \1(w\in S_t)\cdot \frac{\pi(w)}{\pi(S_t)}.
	\]
\end{theorem}

We write $\phi_n = \phi_{p_n}$ and $\psi_n=\psi_{p_n}$, where $p_n$ is the transition matrix at time $n$.

As in~\cite{MorrisPeres} we let 
\[ 
S^{\#}:= \left\{ 
\begin{array}{ccl}
S & \mbox{ if } \pi(S)\le \frac{1}{2} \\
S^c & \mbox{ otherwise }
\end{array} \right.
\]
and 
$$
Z_n:=\frac{\sqrt{\pi(S^{\#}_n)}}{\pi(S_n)}.
$$

The following lemma follows in the same way as in the homogeneous setting of~\cite{MorrisPeres}, but we include the proof for the reader's convenience.

\begin{lemma}\label{lem:zprocess}
	Let $S$ be the Doob transform with respect to absorption at $\Omega$ of the evolving set process associated to a time inhomogeneous Markov chain $X$ with $\prcond{X_{n+1}=x}{X_n=x}{}\geq \gamma$ for all $n$ and~$x$, where $0<\gamma\leq 1/2$. Then for all $n$ and all $S_0\neq \emptyset$ we have 
	\[
	\escondh{Z_{n+1}}{\F_n}{} \leq Z_n \left( 1- \frac{\gamma^2}{2(1-\gamma)^2}\left(\phi_{n+1}\left(\frac{1}{Z_n^{2}}\right)\right)^2 \right),
	\]
	where $\F_n$ stands for the filtration generated by $(S_i)_{i\leq n}$.
\end{lemma}

\begin{proof}[\bf Proof]

Using the transition probability of the Doob transform of the evolving set, we almost surely have 
	\begin{align*}
\escondh{\frac{Z_{n+1}}{Z_n}}{\F_n}{} = \escond{\frac{\pi(S_{n+1})}{\pi(S_n)} \cdot \frac{Z_{n+1}}{Z_n}}{\F_n}{} = \escond{\sqrt{\frac{\pi(S_{n+1}^{\#})}{\pi(S_n^{\#})}}}{\F_n}{}.
\end{align*}
If $\pi(S_n)\leq 1/2$, then 
\begin{align*}
	\escond{\sqrt{\frac{\pi(S_{n+1}^{\#})}{\pi(S_n^{\#})}}}{\F_n}{}\leq \escond{\sqrt{\frac{\pi(S_{n+1}^{})}{\pi(S_n^{})}}}{\F_n}{} \leq 1- \psi_{n+1}(\pi(S_n)).
\end{align*}
Suppose next that $\pi(S_n)>1/2$. Then 
\begin{align*}
	\escond{\sqrt{\frac{\pi(S_{n+1}^{\#})}{\pi(S_n^{\#})}}}{\F_n}{} \leq \escond{\sqrt{\frac{\pi(S_{n+1}^{c})}{\pi(S_n^{c})}}}{\F_n}{} \leq 1- \psi_{n+1}(\pi(S_n^c)).
\end{align*}
Lemma~\ref{lem:phipsi} and the fact that $\phi_{n+1}$ is decreasing now give that 
\begin{align*}
	\escondh{\frac{Z_{n+1}}{Z_n}}{\F_n}{}\leq 1- \frac{\gamma^2}{2(1-\gamma)^2}\cdot (\phi_{n+1}(\pi(S_n)))^2.
\end{align*}
Now note that if $\pi(S_n)\leq 1/2$, then $Z_n = (\pi(S_n))^{-1/2}$. If $\pi(S_n) >1/2$, then $Z_n = \sqrt{\pi(S_n^c)}/\pi(S_n)\leq \sqrt{2}$. Since $\phi_{n+1}(r)=\phi_{n+1}(1/2)$ for all $r>1/2$, we get that we always have
\[
\phi_{n+1}(\pi(S_n)) = \phi_{n+1}\left(\frac{1}{Z_n^{2}}\right)
\]
and this concludes the proof. 
\end{proof}

\section{Preliminaries on supercritical percolation}\label{sec:prelim}

In this section we collect some standard results for supercritical percolation on $\Z_n^d$ that will be used throughout the paper. We write $\B(x,r)$ for the box in $\Z^d$ centred at $x$ of side length $r$. We also use $\B(x,r)$ to denote the obvious subset of $\Z_n^d$ whenever $r<n$. We denote by $\partial \B(x,r)$ the inner vertex boundary of the ball.

\begin{lemma}\label{lem:deterministic}
	Let $A\subseteq \Z_n^d$ be a deterministic set with $|A| = \alpha n^d$, where $\alpha\in (0,1]$. Let $\GG$ be the giant cluster of supercritical percolation in $\Z_n^d$ with parameter $p>p_c$. Then for all $\epsilon\in (0,\theta(p))$ there exists a positive constant $c$ depending on $\epsilon, d, p, \alpha$ so that for all $n$
	\[
	\pr{|A\cap \GG| \notin \left(\alpha (\theta(p)-\epsilon)n^d, \alpha (\theta(p)+\epsilon)n^d \right)} \leq \frac{1}{c} \exp\left(-cn^{\frac{d}{d+1}}\right).
	\]
\end{lemma}

\begin{proof}[\bf Proof]

Let $\beta\in (0,1)$ to be determined later.  
We start the proof by showing that with high probability a certain fraction of the points in $A$ percolate to distance $n^\beta/2$. More precisely, we let 
$A(x) = \{x \leftrightarrow \partial\B(x,n^\beta)\}$. We will first show that for all sets $D\subseteq \Z_n^d$ with $|D|=\gamma n^d$, where $\gamma\in (0,1]$, and for all $\epsilon\in (0,\theta(p))$ there exists  $c>0$ depending on $\epsilon, d, p, \gamma$ so that for all $n$
\begin{align}\label{eq:goalperc}
	\pr{\sum_{x\in D} \1(A(x)) \notin \left(\gamma (\theta(p)-\epsilon)n^d, \gamma (\theta(p)+\epsilon)n^d\right)}\leq \frac{1}{c}\exp\left(-cn^{\frac{d}{d+1}}\right).
\end{align}
Let $\LL$ be a lattice of points contained in $\Z_n^d$ that are at distance $n^\beta$ apart. Then~$\LL$ contains $n^{d(1-\beta)}$ points, and hence there exist $n^{d\beta}$ such lattices. By a union bound over all such lattices $\LL$ we now have
\begin{align}\label{eq:concpro}
	\nonumber\pr{\sum_{x\in D}\1(A(x)) \notin \left(\gamma (\theta(p)-\epsilon)n^d, \gamma (\theta(p)+\epsilon)n^d\right)} \\ \leq n^{\beta d}\cdot \max_{\LL}\pr{\sum_{x\in \LL\cap D}(\1(A(x)))  -  |\LL\cap D|\theta(p)\notin\left(- \gamma\epsilon n^{d(1-\beta)}, \gamma\epsilon n^{d(1-\beta)}\right)}.
\end{align}
Using the standard coupling between bond percolation on the torus and the whole lattice and~\cite[Theorems~8.18 and~8.21]{Grimmett} we get 
\begin{align}
\nonumber	&\pr{A(x)} = \pr{x\in \C_\infty} + \pr{x\notin \C_\infty, x\leftrightarrow \B(x,n^\beta)}\geq  \theta(p) \quad \text{ and }\\
\label{eq:seceq}&\pr{A(x)} \leq  \theta(p)+ e^{-cn^\beta}
\end{align}
for some constant $c$ depending only on $d$ and $p$.
We now fix a lattice $\LL$.
So for all $n$ large enough we can upper bound the probability appearing in~\eqref{eq:concpro} by
\begin{align*}
	\pr{\sum_{x\in \LL\cap D} \left(\1(A(x)) - \pr{A(x)}\right) \notin \left(-\gamma \epsilon n^{d(1-\beta)}, \frac{1}{2}\gamma \epsilon n^{d(1-\beta)} \right)}.
\end{align*}
We now note that for points $x\in \LL\cap D$ the events $A(x)$ are independent. Using a concentration inequality for sums of i.i.d.\ random variables and the fact that $|\LL\cap D|\leq n^{d(1-\beta)}$ we obtain
\begin{align*}
		\pr{\sum_{x\in \LL\cap D} \left(\1(A(x)) - \pr{A(x)}\right) \notin \left(-\gamma \epsilon n^{d(1-\beta)}, \frac{1}{2}\gamma \epsilon n^{d(1-\beta)} \right)}
		\lesssim \exp\left(-cn^{d(1-\beta)} \right),
	\end{align*}
	where $c$ is a positive constant depending on $\gamma$ and $\epsilon$.
	 Plugging this back into~\eqref{eq:concpro} gives
	\begin{align}\label{eq:(A)}
	\pr{\sum_{x\in D}\1(A(x)) \notin \left(\gamma (\theta(p)-\epsilon)n^d, \gamma (\theta(p)+\epsilon)n^d\right)} \lesssim \exp\left(-cn^{d(1-\beta)} \right)
	\end{align}
for a possibly different constant $c$. 

We next turn to prove that 
\begin{align}\label{eq:lowerg}
	\pr{|A\cap \GG|\leq \alpha (\theta(p)-\epsilon)n^d} \lesssim \exp\left(-cn^{\frac{d}{d+1}} \right).
\end{align}
From~\eqref{eq:seceq} and using a union bound we now get 
\begin{align}\label{eq:(B)}
	\pr{\exists \, x\in \B(0,n(1-\delta)): A(x) \cap \{x\centernot\longleftrightarrow \partial\B(0,n)\}}\leq \frac{1}{c}e^{-cn^\beta}.
\end{align}
Using~\cite[Lemma~7.89]{Grimmett} for $d\geq 3$ and duality, the exponential decay in the subcritical regime and the fact that $1-p<1/2$ for $d=2$, we deduce that for all $\delta\in(0,1)$, there exists a constant $c$ (depending on $\delta$, $d$ and $p$) so that for large $n$ and for all $x,y\in \B(0,n(1-\delta))$
\begin{align*}
	\pr{x\leftrightarrow \partial\B(0,n), y\leftrightarrow \partial\B(0,n), x\centernot\longleftrightarrow y}\leq e^{-cn}.
\end{align*}
Using this and a union bound we now get
\begin{align}\label{eq:(C)}
	\pr{\exists \, x,y \in \B(0,n(1-\delta)): \, x\leftrightarrow \partial\B(0,n), y\leftrightarrow \partial\B(0,n), x\centernot\longleftrightarrow y}\leq e^{-cn}.
\end{align}
Take $\til{\epsilon}>0$ and $\delta\in (0,1)$ such that $(1+\theta(p)-\til{\epsilon})(1-\delta)^d>1$. It follows that if there are at least $(\theta(p)-\til{\epsilon})(1-\delta)^d n^d$ points connected to each other in $\B(0,n(1-\delta))$, then the giant cannot be contained in $\B(0,n)\setminus \B(0,n(1-\delta))$. This observation and~\eqref{eq:(A)} (with $D=\B(0,n(1-\delta))$ and so $\gamma  = (1-\delta)^d$ and $\epsilon=\til{\epsilon}$) together with~\eqref{eq:(B)} and~\eqref{eq:(C)} give
\[
	\pr{\exists \, x\in \B(0,(1-\delta)n): \, A(x)\cap \{x\notin \GG\}}\lesssim e^{-cn}+e^{-cn^{\beta}} + e^{-cn^{d(1-\beta)}}.
\]
Taking $\beta=d/(d+1)$ so that $\beta=d(1-\beta)$ we obtain
\begin{align}\label{eq:allpointsconnected}
\pr{\exists \, x\in \B(0,(1-\delta)n): \, A(x)\cap \{x\notin \GG\}}\lesssim e^{-cn^{d/(d+1)}}.
\end{align}
Let now $\til{A} = A\cap \B(0,n(1-\delta))$. Let ${\epsilon'}$ be such that $(\alpha-{\epsilon'})(\theta(p)-{\epsilon'}) = \alpha (\theta(p) -\epsilon)$. 
By decreasing~$\delta$ if necessary we get that $|\til{A}| \geq  (\alpha-{\epsilon'})n^d$. So applying~\eqref{eq:goalperc} we obtain
\begin{align*}
	\pr{\sum_{x\in \til{A}}\1(A(x))\leq (\alpha-{\epsilon'})(\theta(p)-{\epsilon'})n^d} \leq \frac{1}{c}\exp\left(-cn^{\frac{d}{d+1}}\right).
\end{align*}
This together with~\eqref{eq:allpointsconnected} finally gives 
\begin{align*}
	\pr{\sum_{x\in \til{A}} \1(x\in \GG)\leq (\alpha-{\epsilon'})(\theta(p)-{\epsilon'})n^d}\leq \frac{1}{c}\exp\left(-cn^{\frac{d}{d+1}}\right).
\end{align*}
By the choice of ${\epsilon'}$ this proves~\eqref{eq:lowerg}. To finish the proof of the lemma it only remains to show that
\begin{align*}
	\pr{\sum_{x\in A}\1(x\in \GG) \geq \alpha (\theta(p)+\epsilon) n^d} \lesssim \exp\left(-cn^{\frac{d}{d+1}} \right).
\end{align*}
Using~\eqref{eq:goalperc} we can upper bound this probability by
\begin{align*}
	\pr{\exists \, x\in A: \, (A(x))^c\cap \{x\in \GG\}} + \pr{\sum_{x\in A} \1(A(x))\geq \alpha (\theta(p)+\epsilon)n^d} \\
	\leq \pr{{\rm{diam}}(\GG)\leq n^\beta} + \frac{1}{c}\exp\left(-cn^{\frac{d}{d+1}} \right) \lesssim \exp\left(-cn^{\frac{d}{d+1}} \right),
\end{align*}
where the last inequality follows from~\eqref{eq:lowerg} by taking $A=\Z_n^d$.
\end{proof}

\begin{corollary}\label{cor:manyperc}
Let $\GG_1,\GG_2, \ldots$ be the giant components of i.i.d.\ percolation configurations with $p>p_c$ in $\Z_n^d$. Fix $\delta\in (0,1/4)$ and let $k=[2(1-\delta)/(\delta \theta(p))]+1$. Then there exists a positive constant $c$ so that 
\begin{align*}
	\pr{|\GG_1\cup \ldots \cup \GG_k|<(1-\delta)n^d} \leq \frac{1}{c}\exp\left(-cn^{\frac{d}{d+1}}\right).
\end{align*}	
\end{corollary}

\begin{proof}[\bf Proof]
We start by noting that 
\begin{align*}
	|\GG_1\cup \ldots\cup \GG_k| = \sum_{i=1}^{k}|\GG_i\setminus (\GG_1\cup\ldots\cup \GG_{i-1})|,
\end{align*}
where we set $\GG_0=\emptyset$.	Therefore, by the choice of $k$ we obtain
	\begin{align*}
		&\pr{|\GG_1\cup \ldots \cup \GG_k|<(1-\delta)n^d} \\ &\leq  \pr{\exists \, i\leq k: \, |\GG_1\cup\ldots \cup \GG_{i-1}|<(1-\delta)n^d, |\GG_i\setminus(\GG_1\cup\ldots\cup \GG_{i-1})|< \frac{1}{2}\delta \theta(p)n^d}.
	\end{align*}
	For any $i$, since the percolation clusters are independent, by conditioning on $\GG_1,\ldots, \GG_{i-1}$ and using Lemma~\ref{lem:deterministic} we get
	\begin{align*}
		\pr{|\GG_1\cup\ldots \cup \GG_{i-1}|<(1-\delta)n^d, |\GG_i\setminus(\GG_1\cup\ldots\cup \GG_{i-1})|< \frac{1}{2}\delta \theta(p)n^d} \leq \frac{1}{c} \exp\left(-cn^{\frac{d}{d+1}}\right).
	\end{align*}
	Thus by the union bound we obtain
	\begin{align*}
		\pr{|\GG_1\cup \ldots \cup \GG_k|<(1-\delta)n^d} \leq  \frac{k}{c} \exp\left(-cn^{\frac{d}{d+1}}\right) \leq \frac{1}{c'}\exp\left(-c'n^{\frac{d}{d+1}}\right),
		\end{align*}
where $c'$ is a positive constant and this concludes the proof.
\end{proof}

We perform percolation in $\Z_n^d$ with parameter $p>p_c$.
Let $\C_1, \C_2, \ldots$ be the clusters in decreasing order of their size. We write $\C(x)$ for the cluster containing the vertex $x\in \Z_n^d$. For any $A\subseteq \Z_n^d$, we denote by $\diam A$ the diameter of $A$.

\begin{proposition}\label{cl:standardperc}
There exists a constant $c$ so that for all $r$ and for all $n$ we have 
	\begin{align*}
		\pr{\exists i\geq 2: \, \diam{\C_i}\geq r}\leq n^d e^{-cr} + \exp\left(-cn^{\frac{d}{d+1}} \right).
	\end{align*}
\end{proposition}

\begin{proof}[\bf Proof]
We write $\B_r=\B(0,r)$, where as before $\B(0,r)$ denotes the box of side length $r$ centred at~$0$. Then we have 
\begin{align*}
	\pr{{\rm{diam}(\C(0))} \geq r,  \C(0)\neq \C_1} &\leq \pr{0\longleftrightarrow \partial B_{r}, \C(0)\neq \C_1} \\
&\leq  \pr{0\longleftrightarrow \partial \B_n, \C(0)\neq \C_1} + \pr{0\longleftrightarrow \partial \B_{r}, 0 \centernot\longleftrightarrow \partial\B_n }.
\end{align*}
Using the standard coupling between bond percolation on $\Z_n^d$ and bond percolation on $\Z^d$ and~\cite[Theorems~8.18 and~8.21]{Grimmett} we obtain
\begin{align*}
	\pr{0\longleftrightarrow \partial \B_{r}, 0 \centernot\longleftrightarrow \partial\B_n} \leq e^{-cr}.
\end{align*}
Lemma~\ref{lem:deterministic} now gives us that
\begin{align*}
	\pr{\{\C_1\cap \B_{n/4}= \emptyset\}\cup \{\C_1\cap (\B_{n}\setminus \B_{3n/4})=\emptyset\}} \lesssim \exp\left(-cn^{\frac{d}{d+1}} \right).
\end{align*}
So this now implies
\begin{align*}
	\pr{0\longleftrightarrow \partial \B_n, \C(0)\neq \C_1} \lesssim \sum_{x\in 
	\B_{n/4}}\pr{0\longleftrightarrow \partial\B_{3n/4}, x\longleftrightarrow\partial\B_{3n/4}, 0\centernot\longleftrightarrow x} + \exp\left(-cn^{\frac{d}{d+1}} \right).
	\end{align*}
But using~\cite[Lemma~7.89]{Grimmett} we obtain
\begin{align*}
	\pr{0\longleftrightarrow \partial\B_{3n/4}, x\longleftrightarrow\partial\B_{3n/4}, 0\centernot\longleftrightarrow x}\leq e^{-cn}.
\end{align*}
Taking a union bound over all the points of the torus concludes the proof.
\end{proof}

\begin{corollary}\label{cor:giantandsmall}
Consider now dynamical percolation on $\Z_n^d$ with $p>p_c$, where the edges refresh at rate $\mu$, started from stationarity. Let $\C_1(t)$ denote the giant cluster at time $t$.
Then for all $k\in \N$, there exists a positive constant $c$ so that for all $\epsilon<\theta(p)$ we have as~$n\to \infty$
	\begin{align*}
	\pr{|\C_1(t)|\in((\theta(p)-\epsilon) n^d, (\theta(p)+\epsilon) n^d) \,\,{\rm{ and }}\,\,\diam{\C_i(t)}\leq c\log n, \  \forall t\leq n^k/\mu, \, \forall i\geq 2}\to 1.
	\end{align*}
\end{corollary}

\begin{remark}\label{rem:boundonsizefromdiam}
\rm{
Let $\partial A$ denote the edge boundary of a set $A\subseteq \Z^d$. This is how $\partial A$ will be used from now on.
Using then the obvious bound that $|\partial A| \leq (2d)|A|\leq 2d({\rm{diam}}(A))^d$ on the event of Corollary~\ref{cor:giantandsmall} we get that for all $i\geq 2$
\[
|\partial \C_i|\leq 2d|\C_i|\leq 2d(c\log n)^d.
\]
}	
\end{remark}

\section{Hitting the giant component}\label{sec:hitgiant}

In this section we give an upper bound on the time it takes the random walk to hit the giant component. From now on we fix $d\geq 2$ and $p>p_c(\Z^d)$,  and as before $X$ is the random walk on the dynamical percolation process where the edges refresh at rate $\mu$.

%

\emph{Notation}: For every $t>0$ we denote by $\GG_t$ the giant component of the dynamical percolation process $(\eta_t)$ breaking ties randomly. (As we saw~in Corollary~\ref{cor:giantandsmall} with high probability there are no ties in the time interval that we consider.)

\begin{proposition}[Annealed estimates]\label{pro:hittinggiant}
	There exists a stopping time $\sigma$ and $\alpha>0$ such that: 
	\begin{enumerate}
		\item [\rm{(i)}] $\min_{x,\eta_0}\prstart{\frac{11d\log n}{\mu}\leq \sigma\leq \frac{(\log n)^{3d+8}}{\mu}}{x,\eta_0} =1-o(1)$ as $n\to \infty$ and 

\item[\rm{(ii)}] $\min_{x,\eta_0}\prstart{X_{\sigma} \in \GG_\sigma}{x,\eta_0} \geq \alpha$.
	\end{enumerate}

\end{proposition}

\begin{proof}[\bf Proof]

We let $\tau$ be the first time after $11d\log n/\mu$ that $X$ hits the giant component, i.e.
 $$\tau = \inf\left\{t\geq 11d\frac{\log n}{\mu}: \, X_t \in \GG_t\right\}.$$
We now define a sequence of stopping times  by setting 
	$r=2(c\log n)^{d+2}$ for a constant $c$ to be determined, $T_0=0$  and inductively for all $i\geq 0$
\[
T_{i+1} = \inf\left\{t\geq T_i+11d\frac{\log n}{\mu}: \, X_t \notin \B\left(X_{T_i+11d\log n/\mu},r\right) \right\}.
\]
Finally we set $\sigma = \tau\wedge T_{(\log n)^{d+2}}$. We will now prove that $\sigma$ satisfies (i) and (ii) of the statement of the proposition.

{\emph{Proof of} (i).}  By the strong Markov property we obtain for all $n$ large enough and all $x,\eta_1$ 
\begin{align}\label{eq:log11}
\nonumber&\prstart{T_{(\log n)^{d+2}} \leq \frac{(\log n)^{3d+8}}{\mu}}{x,\eta_1} \geq \prstart{T_{i} - T_{i-1} < \log n \cdot \frac{r^2}{\mu}, \, \forall \,1\leq i\leq (\log n)^{d+2}}{x,\eta_1}\\
	&\geq  \left(\min_{x_0,\eta_0}\prcond{T_1 -T_0 < \log n \cdot \frac{r^2}{\mu}}{X_{T_0}=x_0,\eta_{T_0}=\eta_0}{}\right)^{(\log n)^{d+2}}.
	\end{align}
	By~\eqref{eq:hitbigset} of Theorem~\ref{thm:recall} applied to the torus $\Z_{5r}^d$ we get that if $t=c'\cdot r^2/\mu$, where~$c'$ is a positive constant, then starting from any $x_0\in \B(x,r)$ and any bond configuration, the walk exits the ball $\B(x,r)$ by time $t$ with constant probability~$c_1$. Hence the same is true for the process $X$ on $\Z_n^d$ for all starting states $x_0$ and configurations~$\eta_0$.

	 Using this uniform bound over all $\eta_0$ and all $x_0\in \B(x,r)$, we can perform $\log n/c'$ independent experiments to deduce
\[
\prcond{T_1-T_0<\log n\cdot \frac{r^2}{\mu}}{X_{T_0}=x_0,\eta_{T_0}=\eta_0}{} \geq 1- (1-c_1)^{\log n/c'},
\]
and hence substituting this into~\eqref{eq:log11} we finally get
\[
\prstart{T_{(\log n)^{d+2}} \leq \frac{(\log n)^{3d+8}}{\mu}}{x_0,\eta_0}  = 1-o(1) \text{ as } n\to \infty
\]
and this completes the proof of (i).

{\emph{Proof of} (ii).} 
We fix $x,\eta_0$ and we consider two cases:

(1) $\prstart{\tau<T_1}{x,\eta_0}>\frac{1}{(\log n)^{d+2}}$ or

(2) $\prstart{\tau< T_1}{x,\eta_0}\leq \frac{1}{(\log n)^{d+2}}$.

It suffices to prove that under condition (2), there is a  constant $\beta>0$ so that $\prstart{X_{T_1}\in \GG_{T_1}}{x,\eta_0}\geq \beta$. Indeed, this will then imply that 
\begin{align}\label{eq:minbound}
\min_{y,\eta_1}\prstart{\tau\leq T_1}{y,\eta_1} \geq \frac{1}{(\log n)^{d+2}}.
\end{align}
Therefore, in both cases ((1) and (2)) we get that~\eqref{eq:minbound} is satisfied, and hence by the strong Markov property
\begin{align*}
\prstart{\tau> T_{(\log n)^{d+2}}}{x,\eta_0} \leq \left(\max_{y,\eta_1}\prstart{\tau>T_1}{y,\eta_1}  \right)^{(\log n)^{d+2}} = \left(1-\min_{y,\eta_1}\prstart{\tau\leq T_1}{y,\eta_1} \right)^{(\log n)^{d+2}} \leq \frac{1}{e},
\end{align*}
which immediately implies that $\min_{y,\eta_1}\prstart{X_\sigma\in \GG_{\sigma}}{y,\eta_1} \geq 1-e^{-1}$ as claimed.
So we now turn to prove that under (2) there exists a positive constant $\beta$ so that 
\begin{align}\label{eq:goal}
	\prstart{X_{T_1}\in \GG_{T_1}}{x,\eta_0}\geq \beta.
\end{align}
Taking $c$ in the definition of $r$ satisfying $c>50d^2$ we have
\begin{align*}
	\prstart{X_{T_1}\in \GG_{T_1}}{x,\eta_0} 
	\geq \prcond{X_{T_1}\in \GG_{T_1}}{T_1\geq \frac{c\log n}{4d\mu}}{x,\eta_0}  \prstart{T_1\geq \frac{c\log n}{4d\mu}}{x,\eta_0}.
\end{align*}
Since the critical probability for a half-space equals $p_c(\Z^d)$ (as explained right before Theorem 7.35 in \cite{Grimmett}) and by time $\frac{c\log n}{4d\mu}$ all edges in the torus have refreshed after time $T_0$ (with high probability for $c>50d^2$), we infer that, given $ T_1\geq \frac{c\log n}{4d\mu}$, with probability bounded away from 0,  the component of $X_{T_1}$ at time $T_1$ has diameter at least $n/3$. It then follows from Corollary~\ref{cor:giantandsmall} that the first term on the right-hand side of the last display is bounded below by a positive constant.

So it now suffices to prove
\[
\prstart{\tau\geq T_1, T_1\geq \frac{c\log n}{4d\mu}}{x,\eta_0}\geq \beta'>0.
\]
We denote by $\C_t$ the cluster of the walk at time $t$, i.e.\ it is the connected component of the percolation configuration such that $X_t\in \C_t$. Next we define inductively a sequence of stopping times $S_i$ as follows: $S_0=11d\log n/\mu$ and for $i\geq 0$ we let~$S_{i+1}$ be the first time after time $S_i$ that an edge opens on the boundary of $\C_{S_i}$. For all $i\geq 0$ we define 
\[
A_i=\left\{  {\rm{diam}}(\C_{S_i})\leq c\log n\right\} \quad \text{ and } \quad A = \bigcap_{0\leq i\leq (c\log n)^{d+1} -1}A_i.
\]
On the event $A$ we have $T_1\geq S_{(c\log n)^{d+1}}$, since $r=2(c\log n)^{d+2}$ and by the triangle inequality we have for all $i\leq (c\log n)^{d+1} -1$
\begin{align}\label{eq:distance}
	d\left(X_{\frac{11d\log n}{\mu}}, X_{S_{i}}\right) \leq i (c\log n).
\end{align}
We now have
\begin{align}\label{eq:tautau1}
	\prstart{\tau\geq T_1, A^c}{x,\eta_0} = \sum_{0\leq i\leq (c\log n)^{d+1}-1}\prstart{\tau\geq T_1, \cap_{j<i}A_j, A_i^c}{x,\eta_0}.
\end{align}
Note that on the event $\cap_{j<i}A_j\cap \{\tau\geq T_1\}$, we have that $\C_{S_{i}}$ cannot be the giant component, since by time $S_i$ using~\eqref{eq:distance} the random walk has only moved distance at most $i c \log n$ from $X_{11d\log n/\mu}$, and hence cannot have reached the boundary of the box $\B(X_{11d\log n/\mu},r)$. Therefore, choosing $c$ sufficiently large by Proposition~\ref{cl:standardperc} and large deviations for a Poisson random variable we get
\begin{align*}
	\prstart{\tau\geq T_1, \cap_{j<i}A_j, A_i^c}{x,\eta_0} \leq \frac{1}{n},
\end{align*}
and hence plugging this upper bound into~\eqref{eq:tautau1} gives
\begin{align}\label{eq:acompl}	
\prstart{\tau\geq T_1, A^c}{x,\eta_0} \leq\frac{(c\log n)^{d+1}}{n}.
\end{align}
So under the assumption that $\prstart{\tau<T_1}{x,\eta_0}\leq 1/(\log n)^{d+2}$ and~\eqref{eq:acompl} we have for all $n$ sufficiently large
\begin{align}\label{eq:aonly}
	\prstart{A^c}{x,\eta_0} = \prstart{A^c, \tau\geq T_1}{x,\eta_0} + \prstart{A^c, \tau<T_1}{x,\eta_0} \leq
	\frac{2}{(\log n)^{d+2}}.
\end{align}
Setting $Y_i = S_{i}-S_{i-1}$ we now get
\begin{align*}
	\prstart{T_1\geq \frac{c\log n}{4d\mu}}{x,\eta_0} &\geq \prstart{ A, S_{(c\log n)^{d+1}}\geq \frac{c\log n}{4d\mu}}{x,\eta_0} \geq \prstart{\sum_{i=1}^{(c\log n)^{d+1}}Y_i\geq \frac{c\log n}{4d\mu}, A}{x,\eta_0}.
	\end{align*}
	One can define an exponential random variable $E_{(c\log n)^{d+1}}$ with parameter $2d(c\log n)^d\mu$
 such that (1) $Y_{(c\log n)^{d+1}} \ge E_{(c\log n)^{d+1}}$ on
$A_{(c\log n)^{d+1}-1}$ and 
\newline
(2) $E_{(c\log n)^{d+1}}$ is independent of
$\{A_0,\ldots,A_{(c\log n)^{d+1}-1},Y_1\ldots Y_{(c\log n)^{d+1}-1}\}$.
Therefore we deduce
\begin{align*}
	\mathbb{P}_{x,\eta_0}\bigg(\sum_{i=1}^{(c\log n)^{d+1}}&Y_i\geq \frac{c\log n}{4d\mu}, \,A\bigg)\geq \mathbb{P}_{x,\eta_0}\bigg(E_{(c\log n)^{d+1}}+\sum_{i=1}^{(c\log n)^{d+1}-1}Y_i\geq \frac{c\log n}{4d\mu}, \bigcap_{0\leq i<(c\log n)^{d+1}-1}A_i\bigg) \\
	&\quad \quad \quad \quad \quad \quad \quad \quad \quad\quad -\prstart{A_{(c\log n)^{d+1}-1}^c}{x,\eta}\\
	\geq &\prstart{E_{(c\log n)^{d+1}}+\sum_{i=1}^{(c\log n)^{d+1}-1}Y_i\geq \frac{c\log n}{4d\mu}, \bigcap_{0\leq i<(c\log n)^{d+1}-1}A_i}{x,\eta_0} - \frac{2}{(\log n)^{d+2}},
\end{align*}
where for the last inequality we used~\eqref{eq:aonly}. 
Continuing in the same way, for each $i$,
one can define an exponential random variable $E_i$ with parameter $2d(c\log n)^d\mu$
such that (1) $Y_i \geq  E_i$ on $A_{i-1}$
and (2) $E_i$ is independent of
$\{A_0,\ldots,A_{i-1},Y_1,\ldots, Y_{i-1},E_{i+1}, \ldots, E_{(c\log n)^{d+1}}\}$.
We therefore obtain
\begin{align*}
	\prstart{\sum_{i=1}^{(c\log n)^{d+1}}Y_i\geq \frac{c\log n}{4d\mu}, A}{x,\eta_0} \geq \pr{\sum_{i=1}^{(c\log n)^{d+1}} E_i\geq 
	\frac{c\log n}{4d\mu}} - \frac{c'}{\log n},
\end{align*}
where the $E_i$'s are i.i.d.\ exponential random variables of parameter $2d(c\log n)^d\mu$.
By Chebyshev's inequality, we finally conclude that 
\[
\pr{\sum_{i=1}^{(c\log n)^{d+1}} E_i\geq 
	\frac{c\log n}{4d\mu}} = 1-o(1) \quad \text{ as } n\to \infty
\]
and this finishes the proof.
\end{proof}

We now state and prove a lemma that will be used later on in the paper.

\begin{lemma}\label{lem:stay}
	Let $\sigma$ and $\alpha$ be as in the statement of Proposition~\ref{pro:hittinggiant}. Then  as $n\to \infty$ 
	\[
	\min_{x,\eta_0}\prstart{X_t\in \GG_t, \,\, \forall \, t\in \left[\sigma, \sigma+\frac{1}{(\log n)^{d+1}\mu}\right]}{x,\eta_0} \geq \alpha ( 1-o(1)).
	\]
\end{lemma}

\begin{proof}[\bf Proof]

We fix $x,\eta_0$.
From Proposition~\ref{pro:hittinggiant} we have
\begin{align*}
	&\prstart{\exists \, t\in \left[\sigma, \sigma+\frac{1}{(\log n)^{d+1}\mu}\right] :\, X_t\notin \GG_t}{x,\eta_0} \\&=\prstart{X_\sigma\notin \GG_\sigma}{x,\eta_0}
	+ \prstart{X_\sigma\in \GG_\sigma, \exists\, t \in \left(\sigma,\sigma+\frac{1}{(\log n)^{d+1}\mu}\right] :\, X_t\notin \GG_t}{x,\eta_0}\\
	&\leq 1-\alpha + \prstart{X_\sigma\in \GG_\sigma, \exists\, t \in \left(\sigma,\sigma+\frac{1}{(\log n)^{d+1}\mu}\right] :\, X_t\notin \GG_t}{x,\eta_0}.
\end{align*}
Let $\tau$ be the first time that all edges refresh at least once. Thus after time $\tau$ the percolation configuration is sampled according to $\pi_p$. We then have $\prstart{\tau\leq (d+1)\log n/\mu}{x,\eta_0} = 1-o(1)$, and hence from Proposition~\ref{pro:hittinggiant} we get
\[
\prstart{\sigma\geq \tau }{x,\eta_0} = 1-o(1).
\]
This together with Corollary~\ref{cor:giantandsmall} now gives as $n\to\infty$
\begin{align}\label{eq:smallbigsigma}
	\prstart{\forall\, t \in \left[\sigma,\sigma+\frac{1}{(\log n)^{d+1}\mu}\right]: \, |\GG_t|\in (\theta(p) n^d/2, 3\theta(p) n^d/2), \,{\rm{diam}}(\C_i(t))\leq c \log n, \forall \, i\geq 2}{x,\eta_0} \to 1,
\end{align}
where $c$ comes from Corollary~\ref{cor:giantandsmall}.
We now define an event $A$ as follows
\[
A=\left\{ \exists\, t\in \left[\sigma, \sigma+\frac{1}{(\log n)^{d+1}\mu}\right] \text{ and an edge } e:\, d(X_t, e) \leq c\log n \text{ and } e \text{ refreshes at time } t   \right\}.
\]
We also define $B$ to be the event that there exists a time $t\in [\sigma, \sigma+ 1/((\log n)^{d+1}\mu)]$ and an edge~$e$ such that $d(X_t,e)>c\log n$, the edge $e$ updates at time $t$ and this update disconnects $X_t$ from $\GG_t$. Then we have
\[
\prstart{X_\sigma\in \GG_\sigma, \exists \, t\in \left(\sigma, \sigma+\frac{1}{(\log n)^{d+1}\mu}\right]: \, X_t\notin \GG_t}{x,\eta_0}\leq \prstart{X_\sigma \in \GG_\sigma, A}{x,\eta_0} + \prstart{X_\sigma\in \GG_\sigma, B}{x,\eta_0}.
\]
We start by bounding the second probability above. From~\eqref{eq:smallbigsigma} we obtain as $n\to\infty$
\begin{align*}
	\prstart{X_\sigma\in \GG_\sigma, B}{x,\eta_0} \leq \pr{\exists \, t\in \left[\sigma, \sigma+\frac{1}{(\log n)^{d+1}\mu}\right], \exists\, i\geq 2: \, {\rm{diam}}(\C_i(t))\geq c\log n} = o(1).
\end{align*}
It now remains to show that $\prstart{A}{x,\eta_0}=o(1)$ as $n\to\infty$. We now let $\tau_0=\sigma$ and for all $i\geq 1$ we define~$\tau_i$ to be the time increment between the $(i-1)$-st time and the $i$-th time after time~$\sigma$ that either $X$ attempts a jump or an edge within distance $c\log n$ from $X$ refreshes. Then $\tau_i \sim \text{Exp}(1+c_1(\log n)^d\mu)$ for a positive constant $c_1$ and they are independent. These times define a Poisson process of rate $1+c_1(\log n)^d \mu$. 
Using basic properties of exponential variables, the probability that at a point of this Poisson process an edge is refreshed~is $$\frac{c_1(\log n)^d\mu}{1+c_1(\log n)^d\mu}.$$ Therefore, by the thinning property of Poisson processes, the times at which edges within $c\log n$ from $X$ refresh constitute a Poisson process $\mathcal{N}$ of rate $c_1 (\log n)^d\mu$. So we now obtain
\begin{align*}
	\prstart{A}{x,\eta_0} =\pr{\mathcal{N}\left[0,\frac{1}{(\log n)^{d+1}\mu}\right] \geq 1} = 1 - \exp\left(-\frac{c_1}{\log n} \right) = o(1) \text{ as } n\to \infty
\end{align*}
and this concludes the proof.
\end{proof}
%

\section{Good and excellent times}\label{sec:goodtimes}

As we already noted in Remark~\ref{rem:13} we are going to consider the case where $1/\mu > (\log n)^{d+2}$.

We will discretise time by observing the walk $X$ at integer times.  When we fix the environment at all times to be $\eta$, then we obtain a discrete time  Markov chain with time inhomogeneous transition probabilities
\[
p_t^\eta(x,y)=\prcond{X_{t+1}=y}{X_t=x}{\eta} \quad \forall \, x,y\in \Z_n^d, \, t\in \N.
\]
Let $(S_t)_{t\in \N}$ be the Doob transform of the evolving sets associated to this time inhomogeneous Markov chain as defined in Section~\ref{sec:evolve}. Since from now on we will mainly work with the Doob transform of the evolving sets, unless there is confusion, we will write~$\mathbb{P}$ instead of $\widehat{\mathbb{P}}$.

If $G$ is a subgraph of $\Z_n^d$ and $S\subseteq V(G)$, we write $\partial_G S$ for the edge boundary of $S$ in $G$, i.e.\ the set of edges of $G$ with one endpoint of $S$ and the other one in $V(G)\setminus S$.

We note that for every $t$, $\eta_t$ is a subgraph of $\Z_n^d$ with vertex set $\Z_n^d$. 

\begin{definition}\label{def:goodexcellent}
\rm{
	We call an integer time $t$ \emph{good} if $|S_t\cap \GG_t| \geq \tfrac{|S_t|}{(\log n)^{4d+12}}$. We call a good time $t$ \emph{excellent} if 
	\[
	\int_t^{t+1}\left|\partial_{\eta_s}S_t  \right|\,ds \equiv
	\sum_{x\in S_t}\sum_{y\in S_t^c} \int_{t}^{t+1} \eta_s(x,y)\,ds \geq \frac{|\partial_{\eta_t}S_t|}{2},
	\]
	where $\eta_s(x,y)=0$ if $(x,y)\notin E(\Z_n^d)$.
	For all $a \in \N$ we let $G(a)$ and $G_e(a)$ be the set of good and excellent times $t$ respectively with $0\leq t\leq (\log n)^{a}\left(n^2+\tfrac{1}{\mu}\right)$.
	}
\end{definition}
 As we already explained in the Introduction, we will obtain a strong drift for the size of the evolving set at excellent times. So we need to ensure that there are enough excellent times. We start by showing that there is a large number of good times.  More formally we have the following:

 \begin{lemma}\label{lem:goodtimesc}
For all $\gamma\in \N$ and $\alpha>0$,	there exists $n_0$ so that for all $n\geq n_0$, all starting points and configurations $x, \eta_0$ we have
	\[
	\prstart{|G(8d+26+\gamma)|\geq (\log n)^{\gamma}\cdot \left(n^2+\frac{1}{\mu}\right)}{x,\eta_0}  \geq  1 - \frac{1}{n^\alpha}.
	\]
 \end{lemma}

\begin{proof}[\bf Proof]
Fix $\gamma\in \N$ and $\alpha>0$.
To simplify notation we write $G=G(8d+26+\gamma)$. By definition we have 
\[
|G|=\sum_{t=0}^{(\log n)^{8d+26+\gamma}\cdot \left(n^2 +\frac{1}{\mu}\right)} \1\left( \frac{|S_t\cap \GG_t|}{|S_t|} \geq \frac{1}{(\log n)^{4d+12}}   \right). 
\]
For every $i\geq 0$ we define 
\[
J_i=\left[i\cdot(\log n)^{4d+\gamma+12}\cdot\left(n^2+\frac{1}{\mu}\right),(i+1)\cdot(\log n)^{4d+\gamma+12}\cdot\left(n^2+\frac{1}{\mu}\right) \right)\cap \N.
\]
We write $t_i$ for the left endpoint of the interval above. For integer $t$ we let $\F_t$ be the $\sigma$-algebra generated by the evolving set and the environment at integer times up to time $t$.

First of all we explain that for all $x,\eta_0$ and for all $i\geq 0$ we have almost surely
\begin{align}\label{eq:timeingiant}
	\escond{\sum_{t\in J_i}\1(X_t\in \GG_t)}{\F_{t_i}}{x,\eta_0} \geq (\log n)^{\gamma+2}\cdot \left(n^2 + \frac{1}{\mu}\right).
\end{align}
Indeed, in every interval of length $2(\log n)^{3d+8}/\mu$ we have from Proposition~\ref{pro:hittinggiant} and Lemma~\ref{lem:stay} that with constant probability there exists an interval of length $1/((\log n)^{d+1}\mu)$ such that for all $t$ in this interval $X_t\in \GG_t$. Note that since $1/\mu > (\log n)^{d+2}$, this interval has length larger than $1$. This establishes~\eqref{eq:timeingiant}.

Using the coupling of the Doob transform of the evolving set and the random walk given in Theorem~\ref{thm:coupling} we get that 
\begin{align*}
	\prcond{X_t\in \GG_t}{S_t, \GG_t}{} = \frac{|S_t\cap \GG_t|}{|S_t|},
\end{align*}
and hence
\begin{align*}
	|G|=\sum_{i=0}^{(\log n)^{4d+14}-1} T_i := \sum_{i=0}^{(\log n)^{4d+14}-1}\sum_{t\in J_i} \1\left( \prcond{X_t\in \GG_t}{S_t, \GG_t}{x,\eta_0} \geq \frac{1}{(\log n)^{4d+12}} \right).
\end{align*}
For any $x,\eta_0$ we set 
$$A_i(x,\eta_0) := \left\{ t\in J_i: \, \prcond{X_t\in \GG_t}{S_t, \GG_t}{x,\eta_0} \geq \frac{1}{(\log n)^{4d+12}}     \right\}. $$
We now claim that for any $x,\eta_0$ we have almost surely 
\begin{align}\label{eq:claimleb}
	\prcond{|A_i(x,\eta_0)| \geq (\log n)^\gamma\cdot\left(n^2+\frac{1}{\mu}\right)}{\F_{t_i}}{x,\eta_0} \geq \frac{1}{(\log n)^{4d+12}}.
\end{align}
Indeed, if not, then there exists a set $\Omega_0\in \F_{t_i}$ with $\pr{\Omega_0}>0$ such that on $\Omega_0$
\[
\prcond{|A_i(x,\eta_0)| \geq (\log n)^\gamma\cdot\left(n^2+\frac{1}{\mu}\right)}{\F_{t_i}}{x,\eta_0} <\frac{1}{(\log n)^{4d+12}}.
\]
We now define
\begin{align*}
	Y=\sum_{t\in J_i} \prcond{X_t\in \GG_t}{S_t, \GG_t}{x,\eta_0}
\end{align*}
and writing $A_i=A_i(x,\eta_0)$ to simplify notation, we would get on the event $\Omega_0$ that 
\begin{align*}
	\escond{Y}{\F_{t_i}}{x,\eta_0} &= \escond{\sum_{t\in A_i} \prcond{X_t\in \GG_t}{S_t,\GG_t}{x,\eta_0} + \sum_{t\in A_i^c} \prcond{X_t\in \GG_t}{S_t,\GG_t}{x,\eta_0}}{\F_{t_i}}{x,\eta_0} \\&
	\leq \escond{|A_i|}{\F_{t_i}}{x,\eta_0} + \frac{(\log n)^{4d+\gamma+12}}{(\log n)^{4d+12}} \cdot \left(n^2+\frac{1}{\mu}\right)\\ & \leq \frac{(\log n)^{4d+\gamma+12}}{(\log n)^{4d+12}} \cdot \left(n^2+\frac{1}{\mu}\right) + (\log n)^\gamma\cdot\left( n^2+\frac{1}{\mu}\right) + (\log n)^\gamma\cdot\left( n^2+\frac{1}{\mu}\right)\\&=3(\log n)^{\gamma}\cdot\left(n^2+ \frac{1}{\mu}\right).
\end{align*}
But this gives a contradiction for $n\geq e^{\sqrt{3}}$, since we have almost surely
\begin{align*}
	\escond{Y}{\F_{t_i}}{x,\eta_0} &= \escond{\sum_{t\in J_i} \prcond{X_t\in \GG_t}{S_t,\GG_t}{x,\eta_0}}{\F_{t_i}}{x,\eta_0} =\escond{\sum_{t\in J_i} \prcond{X_t\in \GG_t}{S_t,\GG_t, \F_{t_i}}{x,\eta_0}}{\F_{t_i}}{x,\eta_0}\\ &=
\escond{\sum_{t\in J_i}\1(X_t\in \GG_t)}{\F_{t_i}}{x,\eta_0}
	\geq (\log n)^{\gamma+2}\cdot\left(n^2+\frac{1}{\mu}\right),
\end{align*}
where the second equality follows from the Diaconis Fill coupling, the third one from the tower property for conditional expectation and the inequality follows from~\eqref{eq:timeingiant}.

Therefore, since~\eqref{eq:claimleb} holds for all starting points and configurations $x,\eta_0$, we finally conclude that for all $n\geq e^{\sqrt{3}}$,  all $x,\eta_0$ and for all $i$ almost surely
\begin{align*}
	\prcond{T_i\geq (\log n)^\gamma\cdot\left(n^2+\frac{1}{\mu}\right)}{\F_{t_i}}{x,\eta_0} \geq \frac{1}{(\log n)^{4d+12}}.
\end{align*}
Using the uniformity of this lower bound over all starting points and configurations yields for all $n$ sufficiently large and all~$x,\eta_0$ \begin{align*}
	\prstart{|G|\geq (\log n)^{\gamma}\cdot\left(n^2+\frac{1}{\mu}\right)}{x,\eta_0} &\geq \prstart{\exists \, i: \, T_i\geq (\log n)^\gamma\cdot\left(n^2+\frac{1}{\mu}\right)}{x,\eta_0} \\
	&= 1 - \prstart{\forall \, i: \, T_i < (\log n)^\gamma\cdot\left(n^2+\frac{1}{\mu}\right)}{x,\eta_0} \\
	& \geq 1 - \left(1 - \frac{1}{(\log n)^{4d+12}} \right)^{(\log n)^{4d+14}} \geq 1 - \frac{1}{n^\alpha}. 
\end{align*}	
This now finishes the proof.
\end{proof}

Next we show that there are enough excellent times.

\begin{lemma}\label{lem:extra}
For all $\gamma\in \N$ and $\alpha>0$, there exists $n_0$ so that for all $n\geq n_0$ and all~$x,\eta_0$
\[
\prstart{|G_e(8d+26+\gamma)|\geq (\log n)^{\gamma-1} \cdot n^2}{x,\eta_0} \geq  1 - \frac{1}{n^\alpha}.
\]
\end{lemma}

\begin{proof}[\bf Proof]

 For almost every environment, there is an infinite number of good times that we denote by~$t_1,t_2,\ldots$. 
For every good time $t$ we define $I_t$ to be the indicator that $t$ is excellent.

Again to simplify notation we write $G=G(8d+26+\gamma)$ and $G_e= G_e(8d+26+\gamma)$.
Note that if $t$ is good and at least half of the edges of $\partial_{\eta_t}S_t$ do not refresh during $[t,t+1]$, then~$t$ is an excellent time (note that if $\partial_{\eta_t}S_t=\emptyset$, then $t$ is automatically excellent). Let~$E_1,\ldots, E_{|\partial_{\eta_t}S_t|}$ be the first times at which the edges on the boundary~$\partial_{\eta_t}S_t$ refresh. They are independent exponential random variables with parameter $\mu$.

Let $\F_s$ be the $\sigma$-algebra generated by the process (walk, environment and evolving set) up to time~$s$. Then for all $t$, on the event $\{t\in G\}$ we have
\begin{align*}
	\prcond{I_t=1}{\F_t}{x,\eta_0} \geq \prstart{\sum_{i=1}^{|\partial_{\eta_t}S_t|} \1(E_i>1) \geq \frac{|\partial_{\eta_t}S_t|}{2}}{x,\eta_0}.
\end{align*}
Since $\prstart{E_i>1}{x,\eta_0} = e^{-\mu}$ and $\mu\leq 1/2$, there exists $n_0$ so that for all $n\geq n_0$ we have for all $x, \eta_0$
\begin{align*}
	 \prstart{\sum_{i=1}^{|\partial_{\eta_t}S_t|} \1(E_i>1) \geq \frac{|\partial_{\eta_t}S_t|}{2}}{x,\eta_0}\geq \frac{1}{2}.
\end{align*}
Let $A=\{ |G|\geq  (\log n)^{\gamma} \cdot n^2\}$. By Lemma~\ref{lem:goodtimesc} we get $\prstart{A^c}{x,\eta_0} \leq 1/n^\alpha$ for all $n\geq n_0$ and all $x,\eta_0$. Let $G=\{t_1,\ldots, t_{|G|}\}$.  On the event $A$ we have
 \[
 |G_e| \geq \sum_{i=1}^{(\log n)^{\gamma}\cdot n^2} I_{t_i}.
 \]
We thus get for all $x,\eta_0$ and all $n\geq n_0$
\begin{align*}
	\prstart{|G_e|< (\log n)^{\gamma-1} \cdot n^2}{x,\eta_0}& \leq \prstart{|G_e|<  (\log n)^{\gamma-1} \cdot n^2, A}{x,\eta_0} + \frac{1}{n^\alpha} \\
	& \leq \prstart{\sum_{i=1}^{(\log n)^{\gamma}\cdot n^2} I_{t_i} <(\log n)^{\gamma-1} \cdot n^2}{x,\eta_0} + \frac{1}{n^\alpha}.
\end{align*}
Since conditional on the past, the variables $(I_{t_i})_i$ dominate independent Bernoulli random variables with parameter $1/2$, using a standard concentration inequality we get that this last probability decays exponentially in $n$ and this concludes the proof.
\end{proof}

Let $\tau_1, \tau_2,\ldots$ be the sequence of excellent times. Then the previous lemma immediately gives 
\begin{corollary}\label{cor:excellent}
Let $\gamma \in \N$, $\alpha>0$ and $N= (\log n)^{\gamma} \cdot n^2$. Then there exists $n_0$ so that for all~$n\geq n_0$ and all $x,\eta_0$ we have  
\[
\prstart{\tau_N\leq (\log n)^{8d+27+\gamma} \cdot \left( n^2+\frac{1}{\mu}\right)}{x,\eta_0} \geq  1-\frac{1}{n^\alpha}.
\]
\end{corollary}

\section{Mixing times}\label{sec:mixing}

In this section we prove Theorems~\ref{thm:qmixlarge},\ref{thm:qmixcesaro} and Corollary~\ref{cor:hittingtimes}. From now on $d\geq 2$, $p>p_c(\Z^d)$ and $\frac{1}{\mu} >(\log n)^{d+2}$.

\subsection{Good environments and growth of the evolving set}

The first step is to obtain the growth of the Doob transform of the evolving set at excellent times. We will use the following theorem by Gabor Pete~\cite{Pete} which shows that the isoperimetric profile of the giant cluster basically coincides with the profile of the original lattice.

For a subset $S\subseteq \Z_n^d$ we write $S\subseteq \GG$ to denote $S\subseteq V(\GG)$ and we also write $|\GG| = |V(\GG)|$.

\begin{theorem}\label{thm:pete}
{\rm{\cite[Corollary~1.4]{Pete}}}
	For all $d\geq 2$,  $p>p_c(\Z^d)$, $\delta \in (0,1)$ and $c'>0$ there exist $c>0$ and $\alpha>0$ so that for all $n$ sufficiently large
	\[
	\pr{\forall \, S\subseteq \GG: \, S \text{ connected and } c (\log n)^{\frac{d}{d-1}} \leq |S|\leq (1-\delta)|\GG|, \text{ we have } |\partial_{\GG} S|\geq \alpha |S|^{1-\frac{1}{d}}} \geq 1-\frac{1}{n^{c'}}.
	\] 
\end{theorem}

\begin{remark}\rm{
Pete~\cite{Pete} only states that the probability appearing above tends to $1$ as $n\to \infty$, but a close inspection of the proof actually gives the polynomial decay. Mathieu and Remy~\cite{mathieu}
have obtained similar results. 
}
\end{remark}

\begin{corollary}\label{cor:pete}
For all $d\geq 2$, $p>p_c(\Z^d)$, $c'>0$ and $\delta \in (0,1)$ there exist $c>0$ and $\alpha>0$ so that for all $n$ sufficiently large
	\[
	\pr{\forall \, S\subseteq \GG:  c (\log n)^{\frac{d}{d-1}} \leq |S|\leq (1-\delta)|\GG|, \text{ we have } |\partial_{\GG} S|\geq \frac{\alpha |S|^{1-\frac{1}{d}}}{\log n}} \geq 1-\frac{1}{n^{c'}}.
	\]
\end{corollary}

\begin{proof}[\bf Proof]

We only need to prove the statement for all $S$ that are disconnected, since the other case is covered by Theorem~\ref{thm:pete}.
Let $A$ be the event appearing in the probability of Theorem~\ref{thm:pete}.

Let $S$ be a disconnected set satisfying $S\subseteq \GG$ and $c (\log n)^{\frac{d}{d-1}} \leq |S|\leq (1-\delta)|\GG|$. Let $S=S_1\cup \ldots\cup S_k$ be the decomposition of $S$ into its connected components. Then we claim that on the event $A$ we have for all $i\leq k$
\begin{align*}
	|\partial_{\GG} S_i| \geq \alpha\frac{|S_i|^{1-\frac{1}{d}}}{\log n}.
\end{align*}
Indeed, there are two cases to consider: (i) $|S_i|\geq c(\log n)^{d/(d-1)}$, in which case the inequality follows from the definition of the event $A$;
 (ii) $|S_i|<c(\log n)^{d/(d-1)}$, in which case the inequality is trivially true by taking $\alpha$ small in Theorem~\ref{thm:pete}, since the boundary contains at least one vertex.
Therefore we deduce,
\begin{align*}
	|\partial_{\GG} S| =\sum_{i=1}^{k} |\partial_{\GG} S_i| \geq \alpha\sum_{i=1}^{k} \frac{|S_i|^{1-\frac{1}{d}}}{\log n} \geq \alpha\frac{\left(\sum_{i=1}^{k}|S_i|\right)^{1-\frac{1}{d}}}{\log n} = \alpha\frac{|S|^{1-\frac{1}{d}}}{\log n}
\end{align*}
and this completes the proof.
\end{proof}

Recall that for a fixed environment $\eta$ we write $S$ for the Doob transform of the evolving set process associated to $X$ and $\tau_1, \tau_2, \ldots$ are the excellent times as in Definition~\ref{def:goodexcellent} and we take $\tau_0=0$.

\begin{definition}\label{def:goodenv}\rm{
Let~$c_1,c_2$ be two positive constants and $\delta\in (0,1)$. Given $n\geq 1$, define 
\[
t(n)=(\log n)^{16d+47}\cdot (n^2+1/\mu)\quad \text{ and } \quad  N=(\log n)^{8d+20}\cdot n^2.
\]
We call~$\eta$ a~$\delta$-\emph{good} environment if the following conditions hold:

(1) for all $\frac{11d\log n}{\mu}\leq t\leq t(n)\log n$ the giant cluster $\GG_t$ has size $|\GG_t|\in ((1-\delta)\theta(p) n^d, (1+\delta)\theta(p) n^d)$,

(2) for all $\frac{11d\log n}{\mu}\leq t\leq t(n)\log n, \, \forall \, S\subseteq \GG_t$ with 
\[
 c_1(\log n)^{\frac{d}{d-1}}\leq |S| \leq (1-\delta)|\GG_t| \quad \text{ we have } \quad |\partial_{\eta_{t}}S|\geq \frac{c_2 |S|^{1-1/d}}{(\log n)},
 \]

(3) $\prstart{\tau_N\leq t(n)}{x,\eta} \geq 1-\frac{1}{n^{10d}}$ for all $x$,

(4) $\prstart{\tau_N<\infty}{x,\eta}=1$ for all $x$.
}	
\end{definition}

To be more precise we should have defined a $(\delta, c_1, c_2)$-good environment. But we drop the dependence on $c_1$ and $c_2$ to simplify the notation.

\begin{lemma}\label{lem:good}
For all $\delta\in (0,1)$ 
there exist $c_1, c_2, c_3$ positive constants and $n_0\in \N$ such that for all $n\geq n_0$ and all $\eta_0$ we have
 \[
 \prstart{\eta \text{ is $\delta$-good}\,}{\eta_0} \geq 1 - \frac{c_3}{n^{10d}}.
 \]
\end{lemma}

\begin{proof}[\bf Proof]

We first prove that for all $n$ sufficiently large and all $\eta_0$
\begin{align}\label{eq:12conv}
\prstart{\eta \text{ satisfies (1) and (2)}}{\eta_0} \geq 1 -\frac{1}{n^{10d}}.
\end{align}
The number of times that the Poisson clocks on the edges ring between times $11d\log n/\mu$ and~$t(n) \log n$ is a Poisson random variable of parameter at most $d(n^d \mu) \cdot t(n) \log n$. 
 Note that all edges update by time $\frac{11d \log n}{\mu}$ with probability at least
$1-\frac{d}{n^{10d}}$.
Using large deviations for the Poisson random variable, Lemma~\ref{lem:deterministic} and Corollary~\ref{cor:pete} for suitable constants $c$ and $\alpha$ prove~\eqref{eq:12conv}. Corollary~\ref{cor:excellent}, Markov's inequality and a union bound over all $x$ immediately imply
\[
\prstart{\eta \text{ satisfies (3)}}{\eta_0} \geq 1 - \frac{d}{n^{10d}}.
\]
Finally, to prove that $\eta$ satisfies (4) with probability $1$, we note that for almost every environment there will be infinitely many times at which all edges will be open for unit time and so at these times the intersection of the giant component with the evolving set will be large. Therefore such times are necessarily
excellent.
\end{proof}

For all $\delta \in (0,1)$ we now define
\begin{align}\label{eq:taudelta}
\tau_\delta= \inf\{t\in \N: |S_t\cap \GG_t| \geq (1-\delta) |\GG_t|\}.
\end{align}
The goal of this section is to prove the following:

\begin{proposition}\label{pro:90g}
	Let $\delta\in (0,1)$. There exists a positive constant $c$ so that the following holds: for all $n$, if $\eta$ is a $\delta$-good environment, then   for all starting points $x$ we have
	\[
	\prstart{\tau_\delta\leq \tt }{x,\eta} \geq  1 -\frac{c}{n^{10d}}.
	\]
\end{proposition}

Recall from Section~\ref{sec:evolve} the definition of $(Z_k)$ for a fixed environment $\eta$ via
\[
Z_k = \frac{\sqrt{\pi(S_k^{\#})}}{\pi(S_k)}.
\]
Note that we have suppressed the dependence on $\eta$ for ease of notation. The following lemma on the drift of $Z$ using the isoperimetric profile will be crucial in the proof of Proposition~\ref{pro:90g}.

\begin{lemma}\label{lem:supergoodtimes}
Let $\eta$ be a $\delta$-good environment with $\delta\in (0,1)$. Then for all $n$ sufficiently large and for all $1\leq i\leq N$ (recall Definition~\ref{def:goodenv}) we have almost surely
\[
\escondh{Z_{\tau_{i+1}} \1(\tau_\delta\wedge \tt>\tau_{i+1})}{\F_{\tau_i}}{\eta} \leq Z_{\tau_i}\1(\tau_\delta\wedge \tt>\tau_{i}) \left( 1- \left(\phi\left(\frac{1}{Z^2_{\tau_i}}\right)\right)^2\right),
\]
where $\F_t$ is the $\sigma$-algebra generated by the evolving set up to time $t$ and
$(\tau_i)$ is the sequence of excellent times associated to the environment $\eta$ and $\phi$ is defined as
\[
\phi(r) = \begin{cases}
 c \cdot (\log n)^{-\beta}\cdot  n^{-1}\cdot r^{-1/d} \quad &\mbox{if} \quad \frac{(\log n)^{\alpha}}{n^d}\leq r
\leq \frac{1}{2}	\\
	c \cdot n^{-d}\cdot r^{-1}  &\mbox{if} \quad r< \frac{(\log n)^{\alpha}}{n^d} \\
	c \cdot 2^{1/d}\cdot (\log n)^{-\beta}\cdot n^{-1} &\mbox{if} \quad r\in \left[\frac{1}{2},\infty\right)
\end{cases} 
\]	
with $\alpha =4d+12+d/(d-1)$, $\beta= 4d+9-12/d$ and $c$ a positive constant.
\end{lemma}

\begin{proof}[\bf Proof]
Since $\tau_\delta$ is a stopping time, it follows that $\{\tau_\delta\wedge \tt>\tau_i\}\in \F_{\tau_i}$, and hence we obtain
\begin{align}\label{eq:firsteq}
	\escondh{Z_{\tau_{i+1}}\1(\tau_\delta\wedge \tt>\tau_{i+1})}{\F_{\tau_i}} {\eta}\leq \1(\tau_\delta\wedge \tt>\tau_{i})\escondh{Z_{\tau_{i+1}}}{\F_{\tau_i}} {\eta}.
\end{align}
Lemma~\ref{lem:zprocess} implies that $Z$ is a positive supermartingale and since $\eta$ is a $\delta$-good environment, we have  $\tau_N<\infty$ $\mathbb{P}_\eta$-almost surely. We thus get for all $0\leq i\leq N-1$
\begin{align*}
\escondh{Z_{\tau_{i+1}}}{\F_{\tau_i}} {\eta}\leq \escondh{Z_{\tau_{i}+1}}{\F_{\tau_i}}{\eta}.
\end{align*}
Using the Markov property and the fact that $\F_{\tau_i}$ is countably generated gives
\begin{align}\label{eq:codexp}
	\escondh{Z_{\tau_i+1}}{\F_{\tau_i}} {\eta} \leq  \sum_{t,S} \escondh{Z_{t+1}}{\tau_i=t, S_t=S}{\eta} \1(\tau_i=t, S_t=S).
\end{align}
Since $\tau_i$ is a stopping time, the event $\{\tau_i=t\}$ only depends on $(S_{u})_{u\leq t}$. The distribution of $S_{t+1}$ only depends on $S_t$ and the outcome of the independent uniform random variable $U_{t+1}$. Therefore we obtain
\begin{align}\label{eq:n5}
\begin{split}
	\escondh{Z_{t+1}}{\tau_i=t, S_t=S}{\eta} &=\frac{\sqrt{\pi(S^{\#})}}{\pi(S)} \cdot 
	\escondh{\frac{Z_{t+1}}{Z_t}}{S_t=S}{\eta} \\&= \frac{\sqrt{\pi(S^{\#})}}{\pi(S)} \cdot \escond{\sqrt{\frac{\pi(S_{t+1}^{\#})}{\pi(S_t^{\#})}}}{S_t=S}{\eta}, 
	\end{split}
	\end{align}
where for the last equality we used the transition probability of the Doob transform of the evolving set. If $1\leq |S|\leq n^d/2$, then for all $n$ sufficiently large
\begin{align}\label{eq:n1}
	\escond{\sqrt{\frac{\pi(S_{t+1}^{\#})}{\pi(S_t^{\#})}}}{S_t=S}{\eta}\leq\escond{\sqrt{\frac{\pi(S_{t+1})}{\pi(S_t)}}}{S_t=S}{\eta} =1-\psi_{t+1}(S) \leq 1 - \frac{1}{8}\cdot \left(\phi_{t+1}(S)\right)^2,
\end{align}
where the equality is simply the definition of $\psi_{t+1}$ and the last inequality follows from Lemma~\ref{lem:phipsi}, since 
\[
\prcond{X_{t+1}=x}{X_t=x}{\eta} \geq e^{-1}.
\]
 Similarly, if $n^d>|S|>n^d/2$, then, using the fact that the complement of an evolving set process is also an evolving set process, we get
\begin{align}\label{eq:n2}
\escond{\sqrt{\frac{\pi(S_{t+1}^{\#})}{\pi(S_t^{\#})}}}{S_t=S}{\eta}\leq\escond{\sqrt{\frac{\pi(S_{t+1}^c)}{\pi(S_t^c)}}}{S_t=S}{\eta} =1-\psi_{t+1}(S^c) \leq 1 - \frac{1}{8}\cdot \left(\phi_{t+1}(S^c)\right)^2.
\end{align}
Plugging in the definition of $\phi_{t+1}$ we deduce for all $1\leq |S|<n^d$
\begin{align*}
	\phi_{t+1}(S) = \frac{1}{|S|}\sum_{x\in S}\sum_{y\in S^c} \prcond{X_{t+1}=y}{X_t=x}{\eta} &\geq  \frac{1}{2de|S|} \sum_{x\in S} \sum_{y\in S^c} \int_{t}^{t+1}\eta_s(x,y)\,ds\\
	\phi_{t+1}(S^c) = \frac{1}{|S^c|}\sum_{x\in S}\sum_{y\in S^c} \prcond{X_{t+1}=y}{X_t=x}{\eta} &\geq  \frac{1}{2de|S^c|} \sum_{x\in S} \sum_{y\in S^c} \int_{t}^{t+1}\eta_s(x,y)\,ds.
	\end{align*}
Since in~\eqref{eq:codexp} we multiply by $\1(\tau_i=t, S_t=S)$ from now on we take $t$ to be an excellent time, and hence we get from Definition~\ref{def:goodexcellent}
\begin{align}\label{eq:ssc}
\phi_{t+1}(S_t)\geq \frac{1}{4de}\cdot \frac{|\partial_{\eta_{t}}S_t|}{|S_t|}, \quad \phi_{t+1}(S^c_t)\geq \frac{1}{4de}\cdot \frac{|\partial_{\eta_{t}}S_t|}{|S^c_t|} \quad \text{ and } \quad  |S_t\cap \GG_t|\geq \frac{|S_t|}{(\log n)^{4d+12}}.
\end{align}
Since $|\partial_{\eta_t}S_t| \geq |\partial_{\GG_t}S_t| = |\partial_{\GG_t}(\GG_t\cap S_t)|$ we have
\[
\phi_{t+1}(S_t)\geq\frac{1}{4de}\cdot \frac{|\partial_{\GG_{t}}(\GG_t \cap S_t)|}{|S_t|} \quad \text{ and } \quad   \phi_{t+1}(S^c_t)\geq \frac{1}{4de}\cdot \frac{|\partial_{\GG_{t}}(\GG_t \cap S_t)|}{|S^c_t|}.
\]
If $|S_t|\leq c_1(\log n)^{4d+12+d/(d-1)}$, then, since $\eta$ is a $\delta$-good environment, $|\GG_t| \geq (1-\delta)\theta(p)n^d$, and hence,  
 we use the obvious bound
\begin{align}\label{eq:obviousbound}
\phi_{t+1}(S_t) \geq \frac{1}{4de}\cdot\frac{1}{|S_t|}.
\end{align}
Next, if $\frac{n^d}{2}>|S_t|>c_1(\log n)^{4d+12+d/(d-1)}$, then using~\eqref{eq:ssc} and the fact that we are on the event~$\{\tau_\delta\wedge \tt>t\}$ we get that 
$$c_1(\log n)^{d/(d-1)}\leq |\GG_t\cap S_t|\leq (1-\delta)|\GG_t|.$$ 
Therefore, since $\eta$ is a $\delta$-good environment and $t\leq \tt$, (2) of Definition~\ref{def:goodenv} gives that in this case
\begin{align}\label{eq:n3}
	\phi_{t+1}(S_t) \geq \frac{c_2}{4de}  \cdot \frac{|\GG_t\cap S_t|^{1-\frac{1}{d}}}{(\log n)|S_t|} \geq \frac{c}{(\log n)^{4d+9-12/d}}\cdot \frac{|S_t|^{1-\frac{1}{d}}}{|S_t|} = \frac{c}{(\log n)^{4d+9-12/d}}\cdot \frac{1}{|S_t|^{1/d}},
\end{align}
where $c$ is a positive constant and for the second inequality we used~\eqref{eq:ssc} again.

Finally when $|S_t|\geq \frac{n^d}{2}$, on the event $\{\tau_\delta\wedge \tt>t\}$ we have from~\eqref{eq:ssc} and using again (2) of Definition~\ref{def:goodenv}
\begin{align}\label{eq:n4}
\phi_{t+1}(S_t^c) \geq \frac{c}{(\log n)^{4d+9-12/d}} \cdot \frac{|S_t|^{1-\frac{1}{d}}}{n^d - |S_t|}\geq \frac{c \cdot 2^{1/d}}{(\log n)^{4d+9-12/d}} \cdot n^{-1}.
\end{align}
Substituting~\eqref{eq:obviousbound}, \eqref{eq:n3} and~\eqref{eq:n4} into~\eqref{eq:n1} and~\eqref{eq:n2} and then  into~\eqref{eq:firsteq}, \eqref{eq:codexp} and~\eqref{eq:n5} we deduce
\begin{align*}
	\escondh{Z_{\tau_{i+1}}\1(\tau_\delta\wedge \tt>\tau_{i+1})}{\F_{\tau_i}} {\eta}\leq Z_{\tau_i}\1(\tau_\delta\wedge \tt>\tau_{i}) \left(1 - \frac{1}{8}(\phi_{\tau_i+1}(S_{\tau_i}^c))^2 \right) &\1\left(|S_{\tau_i}|\geq \frac{n^d}{2}\right) \\
	+Z_{\tau_i}\1(\tau_\delta\wedge \tt>\tau_{i}) \left(1 - \frac{1}{8}(\phi_{\tau_i+1}(S_{\tau_i}))^2 \right) &\1\left(|S_{\tau_i}|< \frac{n^d}{2}\right)\\
	\leq Z_{\tau_i}\1(\tau_\delta\wedge \tt>\tau_{i}) \left(1 - (\phi(\pi(S_{\tau_i})))^2 \right),
\end{align*}
where the function $\phi$ is given by
\[
\phi(r) = \begin{cases}
 c \cdot (\log n)^{-\beta}\cdot n^{-1}\cdot r^{-1/d} \quad &\mbox{if} \quad \frac{(\log n)^{\alpha}}{n^d}\leq r
\leq \frac{1}{2}	\\
	c\cdot n^{-d}\cdot r^{-1}  &\mbox{if} \quad r< \frac{(\log n)^{\alpha}}{n^d} \\
	c \cdot 2^{1/d}\cdot (\log n)^{-\beta}\cdot  n^{-1} &\mbox{if} \quad r\in \left[\frac{1}{2},\infty\right)
\end{cases}
\]	
with $c$ a positive constant and $\beta= 4d+9-12/d$. We now note that if $\pi(S_t)\leq 1/2$, then $Z_t = (\pi(S_t))^{-1/2}$. If $\pi(S_t) >1/2$, then $Z_t = \sqrt{\pi(S_t^c)}/\pi(S_t)\leq \sqrt{2}$. Since $\phi(r)=\phi(1/2)$ for all $r>1/2$, we get that in all cases
\[
\phi(\pi(S_{\tau_i})) = \phi\left(\frac{1}{Z_{\tau_i}^{2}}\right)
\]
and this concludes the proof.
\end{proof}

\begin{proof}[\bf Proof of Proposition~\ref{pro:90g}]

We define $Y_i=Z_{\tau_i}\1(\tau_\delta\wedge \tt>\tau_i)$ and 
\begin{align*}
	f(y)=\begin{cases}
		\left(\phi\left(\frac{1}{y^2}	\right)\right)^2 \quad&\mbox{ if } y>0 \\
		0 &\mbox{ if } y=0
	\end{cases},
\end{align*}
where $\phi$ is defined in Lemma~\ref{lem:supergoodtimes}.
With these definitions Lemma~\ref{lem:supergoodtimes} gives for all $1\leq i\leq N$
\[
\escondh{Y_{i+1}}{Y_i}{x,\eta} \leq Y_i(1-f(Y_i))
\]
with $Y_1 \leq  n^{d/2}$ for all $n\geq 3$.

Since $\phi$ is decreasing, we get that $f$ is increasing, and hence we can apply~\cite[Lemma~11 (iii)]{MorrisPeres} to deduce that for all $\epsilon>0$ if 
\[
k\geq \int_{\epsilon}^{n^{d/2}} \frac{1}{zf(z)} \, dz,
\]
then we have that 
\[
\estarth{Z_{\tau_k}\1(\tau_\delta\wedge \tt>\tau_k)}{x,\eta} \leq\epsilon.
\]
We now evaluate the integral
\begin{align*}
	\int_{\epsilon}^{n^{d/2}} \frac{1}{zf(z)} \, dz = \int_{\epsilon}^{n^{d/2}} \frac{1}{z(\phi(1/z^2))^2}\,dz = \frac{1}{2}\cdot\int_{\frac{1}{n^d}}^{\frac{1}{\epsilon^2}} \frac{1}{u(\phi(u))^2}\,du.
\end{align*}
Splitting the integral according to the different regions where~$\phi$ is defined and substituting the function we obtain
\begin{align*}
	\int_{\frac{1}{n^d}}^{\frac{1}{\epsilon^2}} \frac{1}{u(\phi(u))^2}\,du \leq c' \cdot n^2 \cdot (\log n)^{2\beta}\cdot \log \frac{1}{\epsilon},
\end{align*}
where $c'$ is a positive constant. Therefore, taking $\epsilon=\frac{1}{n^{10d}}$, this gives that for all $k\geq  c''\cdot n^2 (\log n)^{2\beta+1}$ with $c''=2c'd$ we have that 
\begin{align*}
\estarth{Z_{\tau_k}\1(\tau_\delta\wedge \tt>\tau_k)}{x,\eta} \leq \frac{1}{n^{10d}},
\end{align*}
and hence, since $N=(\log n)^\gamma\cdot n^2$ with $\gamma= 8d+20>2\beta+1$, we deduce
\begin{align}\label{eq:boundonz}
	\estarth{Z_{\tau_N}\1(\tau_\delta\wedge \tt>\tau_N)}{x,\eta} \leq \frac{1}{n^{10d}}.
\end{align}
Clearly we have
\begin{align}\label{eq:inclusion}
	\{\tau_\delta\wedge \tt>\tau_N\}= \{ \pi(S_{\tau_N})\geq 1/2, \tau_\delta\wedge \tt>\tau_N\} \cup \{ \pi(S_{\tau_N})<1/2, \tau_\delta\wedge \tt>\tau_N\}.
\end{align}
For the second event appearing on the right hand side above using the definition of the process $Z$ we get
\begin{align*}
	\{ \pi(S_{\tau_N})<1/2, \tau_\delta\wedge \tt>\tau_N\}\subseteq \{ Z_{\tau_N} \1(\tau_\delta\wedge \tt>\tau_N) > \sqrt{2}\}.
\end{align*}
The first event appearing on the right hand side of~\eqref{eq:inclusion} implies that $|S_{\tau_N}^c|\geq|S_{\tau_N}^c\cap \GG_{\tau_N}| \geq \delta |\GG_{\tau_N}|$. Since $\eta$ is a $\delta$-good environment, by (1) of Definition~\ref{def:goodenv} we have that $|\GG_{\tau_N}|\geq (1-\delta)\theta(p) n^d$. Therefore we obtain
\begin{align*}
\{ \pi(S_{\tau_N})\geq 1/2, \tau_\delta\wedge \tt>\tau_N\} \subseteq \left\{ Z_{\tau_N}\1(\tau_\delta\wedge \tt>\tau_N) \geq \sqrt{\delta(1-\delta)\theta(p)}\right\}.	
\end{align*}
By Markov's inequality and the two inclusions above we now conclude 
\begin{align*}
	\prstart{\tau_\delta\wedge \tt>\tau_N}{x,\eta} &\leq \prstart{Z_{\tau_N} \1(\tau_\delta\wedge \tt>\tau_N) > \sqrt{2}}{x,\eta} \\&\quad \quad + \prstart{Z_{\tau_N}\1(\tau_\delta\wedge \tt>\tau_N) \geq \sqrt{\delta(1-\delta)\theta(p)}}{x,\eta} \\
	&\leq \frac{\estart{Z_{\tau_N}\1(\tau_\delta\wedge \tt>\tau_N)}{x,\eta}}{\sqrt{2}} + \frac{\estart{Z_{\tau_N}\1(\tau_\delta\wedge \tt>\tau_N)}{x,\eta}}{\sqrt{\delta(1-\delta)\theta(p)}} \leq \frac{c}{n^{10d}},
\end{align*}
where $c$ is a positive constant and in the last inequality we used~\eqref{eq:boundonz}. 
Since $\eta$ is a $\delta$-good environment, this now implies that 
\[
\prstart{\tau_\delta\leq \tt}{x,\eta}\geq 1 - \frac{c}{n^{10d}}
\]
and this finishes the proof.
\end{proof}

\subsection{Proof of Theorem~\ref{thm:qmixlarge}}\label{sec:thetalarge}

In this section we prove Theorem~\ref{thm:qmixlarge}. First recall the definition of the stopping time $\tau_\delta$ as the first time $t$ that $|S_t\cap \GG_t|\geq (1-\delta)|\GG_t|$.

\begin{lemma}\label{lem:tau}
Let $p$ be such that $\theta(p)>1/2$. There exists $n_0$ and $\delta>0$ so that for all $n\geq n_0$, if~$\eta$ is a $\delta$-good environment, then for all $x$ 
	\[
	\| \prstart{X_{\tt} \in \cdot}{x,\eta} - \pi \|_{\rm{TV}} \leq  \frac{1 -\delta}{2}.
	\]
\end{lemma}

\begin{proof}[\bf Proof]

Since $\theta(p)>1/2$, there exist $\epsilon> 2\delta>0$ so that 
\begin{align}\label{eq:thetadelta}	
\theta(p)>\frac{1}{2}+2\epsilon\quad  \text{ and } \quad  (1-\delta)^2 \theta(p)>\frac{1}{2}+\epsilon.
\end{align}
Summing over all possible values of $\tau=\tau_\delta$ we obtain
\begin{align}\label{eq:totalvx}
\begin{split}
	&\| \prstart{X_{\tt} \in \cdot}{x,\eta} - \pi \|_{\rm{TV}} = \frac{1}{2} \sum_{z} \left|\prstart{X_{\tt}=z}{x,\eta} - \frac{1}{n^d} \right| \\&\leq  \frac{1}{2} \sum_z \left|\sum_{s\leq \tt} \prstart{X_{\tt}=z, \tau=s}{x,\eta} -\sum_{s\leq \tt}\frac{\prstart{\tau=s}{x,\eta}}{n^d}  \right| +  \prstart{\tau>\tt}{x,\eta}.
\end{split}
\end{align}
By the strong Markov property at time $\tau$ we have 
\begin{align*}
	\prstart{X_{\tt}=z, \tau=s}{x,\eta} &=\sum_{y} \prstart{X_{\tt}=z, \tau=s, X_s =y}{x,\eta} \\
	&= \sum_{y} \prcond{X_{\tt}=z}{X_{s}=y}{x,\eta} \prstart{\tau=s, X_s=y}{x,\eta}.
\end{align*}
Since $\tau$ is a stopping time for the evolving set process, we can use the coupling of the walk and the Doob transform of the evolving set,~Theorem~\ref{thm:coupling}, to get
\begin{align*}
	\prcond{X_s=y}{\tau=s}{x,\eta} = \escond{\frac{\1(y\in S_s)}{|S_s|}}{\tau=s}{x,\eta}.
\end{align*}
For all $s\leq t(n)$ we call $\nu_s$ the probability measure defined by
\[
\nu_s(y) = \escond{\frac{\1(y\in S_s)}{|S_s|}}{\tau=s}{x,\eta}.
\]
We claim that 
\begin{align}\label{eq:claimtv}
\| \nu_s - \pi\|_{\rm{TV}} \leq \frac{1}{2}-\epsilon.
\end{align} 
Indeed, we have 
\begin{align*}
	\| \nu_s - \pi\|_{\rm{TV}} & =\frac{1}{2}\sum_z  \left|\escond{\frac{\1(z\in S_s)}{|S_s|} - \frac{1}{n^d}}{\tau=s}{x,\eta} \right| 
	\leq \escond{\frac{1}{2}\sum_z\left|\frac{\1(z\in S_s)}{|S_s|}
	 - \frac{1}{n^d}\right|}{\tau=s}{x,\eta} \\
	 &= \frac{1}{2}\cdot \escond{1 - \frac{|S_s|}{n^d} + \frac{n^d - |S_s|}{n^d}}{\tau=s}{x,\eta}
	 = \escond{1 - \frac{|S_s|}{n^d}}{\tau=s}{x,\eta}.
\end{align*}
Since $s\leq t(n)$ and $\eta$ is a $\delta$-good environment, we have $|\GG_s|\geq (1-\delta)\theta(p) n^d$, and hence on the event~$\{\tau=s\}$ we get 
\[
|S_s|\geq (1-\delta)^2 \theta(p)n^d > \left(\frac{1}{2} + \epsilon\right) n^d.
\]
This now implies that 
\begin{align*}
	\escond{1 - \frac{|S_s|}{n^d}}{\tau=s}{x,\eta}\leq \frac{1}{2} - \epsilon
\end{align*}
and completes the proof of~\eqref{eq:claimtv}. By the definition of $\nu_s$ we have 
\begin{align*}
&\frac{1}{2} \sum_z \left|\sum_{s\leq \tt} \prstart{X_{\tt}=z, \tau=s}{x,\eta} -\sum_{s\leq \tt}\frac{\prstart{\tau=s}{x,\eta}}{n^d}  \right| \\
	&=\frac{1}{2}\sum_z\left| \sum_{s\leq {\tt}} \sum_y\prcond{X_{\tt}=z}{X_s=y}{x,\eta}  \nu_s(y)\prstart{\tau=s}{x,\eta}-
	\sum_{s\leq {\tt}} \frac{\prstart{\tau=s}{x,\eta}}{n^d}
	\right| \\
	&\leq  \sum_{s\leq {\tt}} \prstart{\tau=s}{x,\eta}\frac{1}{2}\sum_z\left| \sum_y\nu_s(y) \prcond{X_{\tt}=z}{X_s=y}{x,\eta}   - \frac{1}{n^d}
	\right|. 
\end{align*}
But since $\pi$ is stationary for $X$ when the environment is $\eta$, we obtain
\[
\frac{1}{2}\sum_z\left| \sum_y\nu_s(y) \prcond{X_{\tt}=z}{X_s=y}{x,\eta}   - \frac{1}{n^d}
	\right|\leq \|\nu_s - \pi\|_{\rm{TV}}\leq \frac{1}{2}-\epsilon,
\]
where the last inequality follows from~\eqref{eq:claimtv}.
Substituting this bound into~\eqref{eq:totalvx} gives
\begin{align*}
	\| \prstart{X_{\tt} \in \cdot}{x,\eta} - \pi \|_{\rm{TV}} \leq  \frac{1}{2}-\epsilon + \prstart{\tau>\tt}{x,\eta}.
\end{align*}
From Proposition~\ref{pro:90g}
we have 
\begin{align*}
	\prstart{\tau\leq \tt}{x,\eta} \geq  1-\frac{c}{n^{2d}}.
\end{align*}
This together with the fact that we took $2\delta<\epsilon$ finishes the proof.
\end{proof}

\begin{corollary}\label{cor:mixhigh}
Let $p$ be such that $\theta(p)>1/2$. Then there exist $\delta\in(0,1)$ and $n_0$ such that for all $n\geq n_0$ and all starting environments $\eta_0$ we have 
\begin{align*}
	 \cps{(\eta_t)_{t\leq \tt}: \,\forall x,y,\, \tv{P_\eta^{\tt}(x,\cdot) - P_\eta^{\tt}(y,\cdot) } \leq 1-\delta} \geq 1-\delta.
\end{align*}	
\end{corollary}

\begin{proof}[\bf Proof]

Let $\delta$ and $n_0$ be as in the statement of Lemma~\ref{lem:tau}. Then Lemma~\ref{lem:tau} gives that for all $n\geq n_0$, if $\eta$ is a~$\delta$-good environment, then for all $x$ and $y$ we have 
\begin{align*}
	\tv{P_\eta^{\tt}(x,\cdot) - \pi}\leq \frac{1-\delta}{2} \quad \text{ and } \quad	\tv{P_\eta^{\tt}(y,\cdot) - \pi}\leq \frac{1-\delta}{2}. 
\end{align*}
Using this and the triangle inequality we obtain that on the event that $\eta$ is a $\delta$-good environment for all $x$ and $y$
\begin{align*}
	\tv{P_\eta^{\tt}(x,\cdot) -P_\eta^{\tt}(y,\cdot) } \leq 1-\delta.
\end{align*}
Therefore for all $n\geq n_0$ we get for all $\eta_0$
\begin{align*}
	\cps{(\eta_t)_{t\leq \tt}: \,\exists \,x,y,\, \tv{P_\eta^{\tt}(x,\cdot) - P_\eta^{\tt}(y,\cdot) } >1-\delta } \leq \cps{\eta \text{ is not a $\delta$-good environment}}.
\end{align*}
Taking $n_0$ even larger we get from Lemma~\ref{lem:good} that for all $n\geq n_0$
\[
\cps{\eta \text{ is not a $\delta$-good environment}} \leq \delta
\]
and this concludes the proof.
\end{proof}

The following lemma will be applied later in the case where $R$ is a constant or a uniform random variable.

\begin{lemma}\label{lem:timetv}
	Let $R$ be a random time independent of $X$ and such that the following holds: there exists $\delta\in (0,1)$ such that for all starting environments $\eta_0$ we have
	\[
	\cps{\eta: \,\forall x,y,\, \tv{\prstart{X_R=\cdot}{x,\eta} - \prstart{X_R=\cdot}{y,\eta} } \leq 1-\delta} \geq 1-\delta.
	\]
	Then there exists a positive constant $c=c(\delta)$ and $n_0=n_0(\delta)\in \N$ so that if $k=c\log n$ and  $R(k)=R_1+\ldots +R_k$, where $R_i$ are i.i.d.\ distributed as $R$, then for all $n\geq n_0$, all $x,y$ and $\eta_0$ 
	\[
	\cps{\eta:\tv{\prstart{X_{R(k)}=\cdot}{x,\eta} - \prstart{X_{R(k)}=\cdot}{y,\eta} } \leq \frac{1}{n^{3d}}} \geq 1-\frac{1}{n^{3d}}. 
	\]
\end{lemma}

\begin{proof}[\bf Proof]

We fix $x_0, y_0$ and let $X, Y$ be two walks moving in the same environment~$\eta$ and started from $x_0$ and $y_0$ respectively. We now present a coupling of $X$ and $Y$. We divide time into rounds of length $R_1, R_2,\ldots$ and we describe the coupling for every round.

For the first round, i.e.\ for times between $0$ and $R_1$ we use the optimal coupling given by
\[
\prstart{X_{R_1}\neq Y_{R_1}}{x_0,y_0,\eta} = \|\prstart{X_{R_1}=\cdot}{x_0,\eta}
- \prstart{Y_{R_1}=\cdot }{y_0,\eta}\|_{\rm{TV}},
\]
where the environment $\eta$ is restricted between time $0$ and $R_1$. We now change the definition of a good environment. We call $\eta$ a good environment during $[0,R_1]$ if the total variation distance appearing above is smaller than $1-\delta$.

If $X$ and $Y$ did not couple after $R_1$ steps, then they have reached some locations $X_{R_1}=x_1$ and $Y_{R_1} = y_1$. In the second round we couple them using again the corresponding optimal coupling, i.e.
\[
\prstart{X_{R_2}\neq Y_{R_2}}{x_1,y_1,\eta} = \|\prstart{X_{R_2}=\cdot}{x_1,\eta}
- \prstart{Y_{R_2}=\cdot }{y_1,\eta}\|_{\rm{TV}}.
\]
Similarly we call $\eta$ a good environment for the second round if the total variation distance above is smaller than $1-\delta$. We continue in the same way for all later rounds. By the assumption on $R$, i.e.\ the bound on the probability given in the statement of the lemma is uniform over all starting points $x$ and $y$ and the initial environment, we get that for all~$\eta_0$
\[
\cps{\eta \text{ is good for the $i$-th round}}  \geq 1-\delta
\]
and the same bound is true even after conditioning on the previous $i-1$ rounds.
Let $k=c\log n$ for a constant $c$ to be determined. Let $E$ denote the number of good environments in the first $k$ rounds. We now get
\begin{align*}
	\prstart{X_{R(k)}\neq Y_{R(k)}}{x_0,y_0, \eta_0} \leq \prstart{E\leq \frac{(1-\delta)k}{2}}{x_0,y_0, \eta_0}+ \prstart{E > \frac{(1-\delta)k}{2}, X_{R(k)}\neq Y_{R(k)}}{x_0,y_0, \eta_0}.
\end{align*}
By concentration, since we can stochastically dominate $E$ from below by ${\rm{Bin}}(k,1-\delta)$, the first probability decays exponentially in $k$. For the second probability, on the event that there are enough good environments, since the probability of not coupling in each round is at most $1-\delta$, by successive conditioning we get
\[
\prstart{E > \frac{(1-\delta)k}{2}, X_{R(k)}\neq Y_{R(k)}}{x_0,y_0, \eta_0} \leq (1-\delta)^{(1-\delta)k/2}.
\]
Therefore, taking $c=c(\delta)$ sufficiently large we get overall for all $n$ sufficiently large
\[
\prstart{X_{R(k)}\neq Y_{R(k)}}{x_0,y_0, \eta_0} \leq \frac{1}{n^{6d}}.
\]
So by Markov's inequality again we obtain for all $n$ sufficiently large
\begin{align*}
	&\cps{\eta: \|\prstart{X_{R(k)}=\cdot}{x_0,\eta}
- \prstart{Y_{R(k)}=\cdot }{y_0,\eta}\|_{\rm{TV}} > \frac{1}{n^{3d}}} \\
&\leq
     n^{3d}\cdot  \mathcal{E}_{\eta_0}\left[\|\prstart{X_{R(k)}=\cdot}{x_0,\eta}
- \prstart{Y_{R(k)}=\cdot }{y_0,\eta}\|_{\rm{TV}}\right]\\&
     \leq n^{3d}\cdot \prstart{X_{R(k)}\neq Y_{R(k)}}{x_0,y_0,\eta_0}\leq 
	\frac{1}{n^{3d}},
\end{align*}
where $\mathcal{E}$ is expectation over the random environment. This finishes the proof.
\end{proof}

\begin{proof}[\bf Proof of Theorem~\ref{thm:qmixlarge}]

Let $R=\tt$. Then by Corollary~\ref{cor:mixhigh} there exists $n_0$ such that $R$ satisfies the condition of Lemma~\ref{lem:timetv} for $n\geq n_0$. So applying Lemma~\ref{lem:timetv} we get for all $n$ sufficiently large and
 all $x_0,y_0$ and $\eta_0$ 
 \[
\cps{\eta: \|\prstart{X_{k\tt}=\cdot}{x_0,\eta}
- \prstart{Y_{k\tt}=\cdot }{y_0,\eta}\|_{\rm{TV}} > \frac{1}{n^{3d}}} \leq \frac{1}{n^{3d}},
\]
where $k=c\log n$.
By a union bound over all starting states $x_0,y_0$ we deduce
\begin{align*}
	\cps{\eta: \max_{x_0,y_0}\|P_\eta^{k\tt}(x_0,\cdot) - P_\eta^{k\tt}(y_0,\cdot) \|_{\rm{TV}} > \frac{1}{n^{3d}}} \leq n^{2d} \cdot \frac{1}{n^{3d}} = \frac{1}{n^d}.
\end{align*}
This proves that for all $n$ sufficiently large
\[
\cps{\eta:  \tmixx{n^{-3d}}{\eta} \geq k\tt } \leq n^{-d}
\]
and thus completes the proof of the theorem.
\end{proof}

\subsection{Proof of Theorem~\ref{thm:qmixcesaro}}\label{sec:thetasmall}

\begin{proof}[\bf Proof of Theorem~\ref{thm:qmixcesaro}]
Let $\delta=\epsilon/100$ and $k=[2(1-\delta)/(\delta \theta(p))]+1$.
For every starting point $x_0$ we are going to define a sequence of stopping times. First let~$\xi_1$ be the first time that all the edges refresh at least once. Let $\til{\delta}=\delta/k$.
Then we define $\tau_1=\tau_1(x_0)$ 
\[
\tau_1 = \inf\left\{t\geq \xi_1: |S_t\cap \GG_t| \geq (1-\til{\delta})|\GG_t|\right\}\wedge (\xi_1+\tt),
\]
where $(S_t)$ is the evolving set process starting at time $\xi_1$ from $\{X_{\xi_1}\}$ and coupled with $X$ using the Diaconis Fill coupling.
We define inductively, $\xi_{i+1}$ as the first time after $\xi_i+\tt$ that all edges refresh at least once. In order to now define $\tau_{i+1}$, we start a new evolving set process which at time $\xi_{i+1}$ is the singleton $\{X_{\xi_{i+1}}\}$. (This new restart does not affect the definition of the earlier $\tau_j$'s.) To simplify notation, we call this process  again $S_t$ and we couple it with the walk $X$ using the Diaconis Fill coupling. Next we define
\[
\tau_{i+1}=\inf\left\{t\geq \xi_{i
+1}: |S_t\cap \GG_t| \geq (1-\til{\delta})|\GG_t|\right\}\wedge (\xi_{i+1}+\tt).
\]
From now on we call $\eta$ a \emph{good} environment if $\eta$ is a $\delta$-good environment and $\xi_k\leq 2k\tt$. Lemma~\ref{lem:good} and the definition of the $\xi_i$'s give for all $\eta_0$
\begin{align}\label{eq:notlemma}
	\mathcal{P}_{\eta_0}(\eta\text{ is good})  \geq 1-\frac{c_4}{n^{10d}},
\end{align}
where $c_4$ is a positive constant.
By Proposition~\ref{pro:90g} there exists a positive constant $c$ so that if $\eta$ is a good environment, then for all $x_0$ and for all $1\leq i\leq k$ we have

\begin{align}\label{eq:important}
	\prstart{\tau_{i}-\xi_{i}\leq \tt}{x_0,\eta} \geq 1 -\frac{c}{n^{10d}}.
\end{align}
We  will now prove that there exists a positive constant $c'$ so that for all $x_0$
\begin{align}\label{eq:claimclusters}
	\prstart{|\GG_{\tau_1}\cup\ldots\cup\GG_{\tau_k}|<(1-\delta)n^d}{x_0, \eta_0}\leq \frac{c'}{n^{2d}}.
\end{align}
Writing again $\mathcal{E}$ for expectation over the random environment and using~\eqref{eq:notlemma} and~\eqref{eq:important} we obtain for all $i\leq k$ that there exists a positive constant $c''$ so that for all $n$ sufficiently large and for all $x_0, \eta_0$
\begin{align*}
	\prstart{\tau_{i}-\xi_{i}\leq \tt}{x_0, \eta_0} \geq \mathcal{E}_{\eta_0}\left[\prstart{\tau_{i}-\xi_{i}\leq \tt}{x_0,\eta}\1(\eta\text{ is good})\right] \geq 1-\frac{c''}{n^{10d}}.
	\end{align*}
	This and Markov's inequality now give that for all $n$ sufficiently large
	\begin{align}\label{eq:taudtaui}
	\cps{\eta: \forall \,x_0, \, \prstart{\tau_k\leq (\log n)\tt}{x_0,\eta} \geq 1-\frac{1}{n^{{2d}}}}\geq 1-\frac{c''}{n}.
\end{align}
Since every edge refreshes after an exponential time of parameter $\mu$, it follows that the number of different percolation clusters that appear in an interval of length $t$ is stochastically bounded by a Poisson random variable of parameter $\mu\cdot t\cdot dn^d$. Therefore, the number of possible percolation configurations in the interval $[\xi_i, \xi_i+\tt]$ is dominated by a Poisson variable $N_i$ of parameter $\mu\cdot \tt\cdot dn^d$. By the concentration of the Poisson distribution, we obtain\begin{align*}
	\pr{\exists \, i\leq k: N_i\geq n^{d+4}} \leq \exp\left(-c_1n \right),
\end{align*} 
where $c_1$ is another positive constant. Let $\GG^1, \ldots, \GG^k$ be the giant components of independent supercritical percolation configurations. Since the percolation clusters obtained at the times $\xi_i$ are independent, using Corollary~\ref{cor:manyperc} in the third inequality below we deduce that for all $n$ sufficiently large
\begin{align*}
	&\prstart{|\GG_{\tau_1}\cup\ldots\cup\GG_{\tau_k}|<(1-\delta)n^d}{x_0, \eta_0} \\
	&\leq \prstart{|\GG_{\tau_1}\cup\ldots\cup\GG_{\tau_k}|<(1-\delta)n^d, \{\tau_{i}-\xi_{i}\leq \tt\} \cap \{N_i\leq n^{d+4}\}, \forall \, i\leq k}{x_0,\eta_0} + e^{-c_1n} + k\frac{c}{n^{10d}}\\&\leq 	
	n^{(d+4)k} \pr{|\GG^1\cup\ldots \cup \GG^k|<(1-\delta)n^{d}} +  k \frac{2c}{n^{10d}}\leq  \frac{n^{(d+4)k}}{c}\exp\left(-cn^{\frac{d}{d+1}} \right)  + k\frac{2c}{n^{10d}}\leq \frac{c''}{n^{10d}},
\end{align*}
where $c''$ is a positive constant uniform for all $x_0$ and $\eta_0$. 
This proves~\eqref{eq:claimclusters}. So we can sum this error over all starting points $x_0$ and get using Markov's inequality that for all $n$ sufficiently large and all~$\eta_0$
\begin{align}\label{eq:clind}
	\cps{\eta: \forall x_0, \,\prstart{|\GG_{\tau_1}\cup\ldots\cup\GG_{\tau_k}|\geq (1-\delta)n^d}{x_0,\eta}\geq 1-\frac{1}{n}}\geq 1-\frac{c'}{n}.
\end{align}
The definition of the stopping times $\tau_i$ immediately yields
\[
\{|\GG_{\tau_1}\cup\ldots\cup \GG_{\tau_k}| \geq (1-\delta)n^d\} \subseteq \{|S_{\tau_1}\cup \ldots \cup S_{\tau_k}|\geq (1-\delta)^2n^d\}.
\]
This together with~\eqref{eq:clind} now give
\begin{align}\label{eq:onestar}
	\cps{\eta: \forall x_0, \,\prstart{|S_{\tau_1}\cup\ldots\cup S_{\tau_k}|\geq (1-\delta)^2 n^d}{x_0,\eta}\geq 1-\frac{1}{n}}\geq 1-\frac{c'}{n}.
\end{align} 
Remember the dependence on $x_0$ of the stopping times $\tau_i$ that we have suppressed. We now change the definition of a good environment and call $\eta$ \emph{good} if it satisfies the following for all $x_0$
\begin{align}\label{eq:etagood}
\prstart{|S_{\tau_1}\cup\ldots\cup S_{\tau_k}|\geq (1-\delta)^2 n^d}{x_0,\eta}\geq 1-\frac{1}{n}
\quad\text{ and } \quad  \prstart{\tau_k\leq (\log n)\tt}{x_0,\eta} \geq 1-\frac{1}{n^{2d}}
\end{align}
From~\eqref{eq:taudtaui} and~\eqref{eq:onestar} we get that for all $\eta_0$
\begin{align}\label{eq:twostars}
\cps{\eta\text{ is good}} \geq 1-\frac{c'+c''}{n}.
\end{align}
We now define a stopping time $\tau(x_0)$ by selecting $i\in \{1,\ldots, k\}$ uniformly at random and setting $\tau(x_0)=\tau_i(x_0)$. Then at this time we have 
\begin{align*}
	\prstart{X_{\tau(x_0)} = x}{x_0,\eta} = \sum_{i=1}^{k}\frac{1}{k}\prstart{X_{\tau_i}=x}{x_0,\eta} = \frac{1}{k}\sum_{i=1}^{k} \estart{\frac{\1(x\in S_{\tau_i})}{|S_{\tau_i}|}}{x_0,\eta} \geq \frac{1}{kn^d}\prstart{x\in S_{\tau_1}\cup\ldots \cup S_{\tau_k}}{x_0,\eta}.
\end{align*}
We now set $f_1(x) = \prstart{x\in S_{\tau_1}\cup\ldots \cup S_{\tau_k}}{x_0,\eta}$ for all $x$. Since $\eta$ is a good environment, then for some $\delta'<\epsilon/50$ we have for all $n$ sufficiently large
\begin{align}\label{eq:f1sum}
\sum_x f_1(x) = \estart{|S_{\tau_1}\cup\ldots\cup S_{\tau_k}|}{x_0,\eta} \geq (1-\delta)^2n^d \left(1-\frac{1}{n}\right)= (1-\delta') n^d.
\end{align}
First let $c=c(\epsilon)\in \N$ be a constant to be fixed later. 
In order to define the stopping rule, we first repeat the above construction $ck$ times. More specifically, when $X_0=x_0$, we let $\sigma_1= \tau(x_0)\wedge (\log n)\tt$. Then, since $\eta$ is a good environment, we obtain
\[
\prstart{X_{\sigma_1}=x}{x_0,\eta} \geq \frac{1}{kn^d} f_1(x) - \frac{1}{n^{2d}}.
\]

 Let $X_{\sigma_1} = x_1$. Then we define in the same way as above a stopping time $\tau(x_1)$ with the evolving set process starting from $\{x_1\}$ and the environment considered after time $\sigma_1$. Then we set $$\sigma_2 = \sigma_1+(\tau(x_1)-\sigma_1)\wedge (\log n)\tt.$$ We continue in this way and define a sequence of stopping times $\sigma_i$ for all $i< ck$. 
In the same way as for the first round for all $i< ck$ we have 
\[
\prstart{X_{\sigma_i}=x}{x_0,\eta} \geq \frac{1}{kn^d} f_i(x) - \frac{1}{n^{2d}}
\]
and the function $f_i$ satisfies~\eqref{eq:f1sum}.

We next define the stopping rule. To do so we will explain what is the probability of stopping in every round. We define the set $A_1$ of good points for the first round as follows: 
\[
A_1 = \left\{ x: \, \prstart{X_{\sigma_1}=x}{x_0, \eta} \geq \frac{1}{2kn^d}\right\}.
\]
We now sample $X$ at time $\sigma_1$. If $X_{\sigma_1}=x\in A_1$, then at this time we stop with probability 
\[
\frac{1}{2kn^d \prstart{X_{\sigma_1}=x}{x_0,\eta}}.
\]
If we stop after the first round, then we set $T=\sigma_1$.
So if $x\in A_1$, we have 
\begin{align*}
	\prstart{X_{T}=x, T=\sigma_1}{x_0,\eta} = \frac{1}{2kn^d}.
\end{align*}
From~\eqref{eq:f1sum} we get that $|A_1|\geq (1-3\delta')n^d$ for all $n$ sufficiently large. Therefore, summing over all $x\in A_1$ we get that 
\[
\prstart{T=\sigma_1}{x_0,\eta} \geq \frac{1-3\delta'}{2k}.
\]
Therefore, this now gives for $x\in A_1$
\begin{align*}
	\prcond{X_T=x}{T=\sigma_1}{x_0,\eta} \leq \frac{1}{(1-3\delta')n^d}.
\end{align*}
We now define inductively the probability of stopping in the $i$-th round. Suppose we have not stopped up to the $i-1$-st round. We define the set of good points for the $i$-th round via
\[
A_i=\left\{x: \, \prstart{X_{\sigma_i}=x}{x_0,\eta}\geq \frac{1}{2kn^d}   \right\}
\]
If $X_{\sigma_i}=x\in A_i$, then the probability we stop at the $i$-th round is 
\[
\frac{1}{2kn^d \prstart{X_{\sigma_i}=x}{x_0,\eta}}
\]
and as above we obtain by summing over all $x\in A_i$ and using that $|A_i|\geq (1-3\delta')n^d$
\[
\prcond{T=\sigma_i}{T>\sigma_{i-1}}{x_0,\eta} \geq \frac{1-3\delta'}{2k} \quad\text{ and }\quad  \prcond{X_T=x}{T=\sigma_i}{x_0,\eta} \leq \frac{1}{(1-3\delta')n^d}, \,\, \forall \, x\in A_i.
\]
If we have not stopped before the $ck$-th round, then we set $T=\sigma_{ck+1}$.
Notice, however, that 
\[
\prstart{T=\sigma_{ck+1}}{x_0,\eta} \leq \left(1-\frac{1-3\delta'}{2k} \right)^{ck}\leq  \exp\left(-c(1-3\delta') \right).
\]
For every round $i\leq ck$, we now have that 
\begin{align*}
\norm{\prcond{X_T=\cdot}{T=\sigma_i}{x_0,\eta} - \pi}
_{\rm{TV}} &=\sum_{x\in A_i} \left(\prcond{X_T=x}{T=\sigma_i}{x_0,\eta} - \frac{1}{n^d} \right)_+ + \frac{|A_i^c|}{n^d} \\
&\leq \sum_{x\in A_i}\left( \frac{1}{(1-3\delta')n^d} - \frac{1}{n^d} \right) + 3\delta' \leq \frac{3\delta'}{1-3\delta'}+3\delta' \leq 10\delta',
\end{align*}
since $\epsilon <1/4$.
So we now get overall 
\begin{align*}
	\norm{\prstart{X_T=\cdot}{x_0,\eta}- \pi}_{\rm{TV}} &\leq \sum_{i\leq ck} \prstart{T=\sigma_i}{x_0,\eta} \norm{\prcond{X_{T}=\cdot}{T=\sigma_i}{x_0,\eta}-\pi}_{\rm{TV}} + \prstart{T=\sigma_{ck+1}}{x_0,\eta}\\
	&\leq 10\delta' + \exp\left(-c(1-3\delta') \right).
\end{align*}
We now take $c=c(\epsilon)$ so that the above bound is smaller than $\epsilon$.
Finally, by the definition of the stopping times $\sigma_i$, we also get that $\estart{T}{x_0,\eta}\leq c k (\log n)\tt$ and this concludes the proof.
\end{proof}

\begin{proof}[\bf Proof of Corollary~\ref{cor:hittingtimes}]
Let $n=10r$. It suffices to prove the statement of the corollary for $X$ being a random walk on dynamical percolation on $\Z_n^d$.
	From Theorem~\ref{thm:qmixcesaro} there exists $a$ so that for all $n$ large enough and all $x$ and $\eta_0$
	\[
	\prstart{\exists \, t\leq \left(n^2+\frac{1}{\mu} \right)(\log n)^a: \, \norm{X_t}\geq r}{x,\eta_0}\geq \frac{1}{2}.
	\]
	The statement of the corollary follows by iteration.
\end{proof}

\section*{Acknowledgements}

We thank Sam Thomas for a careful reading of the manuscript and providing a number of useful comments. We also thank Microsoft Research for its hospitality where parts of this work were completed. The third author also acknowledges the support of the Swedish Research Council and the Knut and Alice Wallenberg Foundation.

\bibliographystyle{abbrv}
\bibliography{biblio}

\begin{thebibliography}{1}

\bibitem{SebAn}
S.~Andres, A.~Chiarini, J.-D. Deuschel, and M.~Slowik.
\newblock Quenched invariance principle for random walks with time-dependent
  ergodic degenerate weights.
\newblock {\em Ann. Probab.}
\newblock to appear.

\bibitem{DiaconisFill}
P.~Diaconis and J.~A. Fill.
\newblock Strong stationary times via a new form of duality.
\newblock {\em Ann. Probab.}, 18(4):1483--1522, 1990.

\bibitem{Grimmett}
G.~Grimmett.
\newblock {\em Percolation}, volume 321 of {\em Grundlehren der Mathematischen
  Wissenschaften [Fundamental Principles of Mathematical Sciences]}.
\newblock Springer-Verlag, Berlin, second edition, 1999.

\bibitem{LevPerWil}
D.~A. Levin, Y.~Peres, and E.~L. Wilmer.
\newblock {\em Markov chains and mixing times}.
\newblock American Mathematical Society, Providence, RI, 2009.
\newblock With a chapter by James G. Propp and David B. Wilson.

\bibitem{mathieu}
P.~Mathieu and E.~Remy.
\newblock Isoperimetry and heat kernel decay on percolation clusters.
\newblock {\em Ann. Probab.}, 32(1A):100--128, 2004.

\bibitem{MorrisPeres}
B.~Morris and Y.~Peres.
\newblock Evolving sets, mixing and heat kernel bounds.
\newblock {\em Probab. Theory Related Fields}, 133(2):245--266, 2005.

\bibitem{quenched-sub}
Y.~Peres, P.~Sousi, and J.~E. Steif.
\newblock Quenched exit times for random walk on dynamical percolation.
\newblock 2017.
\newblock arXiv.

\bibitem{PerStaufSteif}
Y.~Peres, A.~Stauffer, and J.~E. Steif.
\newblock Random walks on dynamical percolation: mixing times, mean squared
  displacement and hitting times.
\newblock {\em Probab. Theory Related Fields}, 162(3-4):487--530, 2015.

\bibitem{Pete}
G.~Pete.
\newblock A note on percolation on {$\Bbb Z\sp d$}: isoperimetric profile via
  exponential cluster repulsion.
\newblock {\em Electron. Commun. Probab.}, 13:377--392, 2008.

\end{thebibliography}

\end{document}